\documentclass[a4paper, 10pt, DIV = 13]{scrartcl}
\usepackage[utf8]{inputenc}
\usepackage[english]{babel}
\usepackage{amsmath,amsthm}
\usepackage{amsfonts}
\usepackage{amssymb}
\usepackage{lmodern}
\usepackage{hyperref}
\usepackage{xcolor}
\usepackage{hyperref}
\usepackage{parskip}

\usepackage{mathtools}
\mathtoolsset{showonlyrefs}

\usepackage{bm} 
\hypersetup{breaklinks=true}

\usepackage{mathtools}

\definecolor{darkred}{rgb}{0.5,0,0}
\definecolor{darkgreen}{rgb}{0,0.5,0}
\definecolor{darkblue}{rgb}{0,0,0.5}
\hypersetup{colorlinks,
linkcolor=darkblue,
filecolor=darkgreen,
urlcolor=darkred,
citecolor=darkblue}

\makeatletter
\def\namedlabel#1#2{\begingroup
    #2%
    \def\@currentlabel{#2}%
    \phantomsection\label{#1}\endgroup
}
\makeatother

\numberwithin{equation}{section}

\usepackage{enumitem}
\setlist{nosep}
\setlist{noitemsep}

\usepackage{mathtools}

\newtheorem{theorem}{Theorem}

\newtheorem{proposition}{Proposition}[section]
\newtheorem{lemma}[proposition]{Lemma}
\newtheorem{corollary}[proposition]{Corollary}

\newtheorem{claim}[proposition]{Claim}

\theoremstyle{definition}
\newtheorem{definition}[proposition]{Definition}
\newtheorem{remark}[proposition]{Remark}

\newcommand{\R}{\mathbb{R}}

\DeclareMathOperator{\supp}{supp}
\DeclareMathOperator{\Id}{Id}
\DeclareMathOperator{\dist}{dist}
\DeclareMathOperator{\dive}{div}

\newcommand{\0}{\mathsf{0}}
\newcommand{\1}{\mathsf{1}}
\newcommand{\2}{\mathsf{2}}

\newcommand{\hal}{\frac{1}{2}}

\newcommand{\Var}{\mathrm{Var}}

\renewcommand{\u}{\mathsf{1}}
\newcommand{\dd}{\mathtt{d}}

\newcommand{\DD}{\mathfrak{D}}
\newcommand{\RN}{\mathrm{R}_N}

\newcommand{\Conf}{\mathrm{Conf}}

\newcommand{\Fluct}{\mathrm{Fluct}}

\newcommand{\LN}{\Sigma_N}
\newcommand{\XN}{\mathrm{X}_N}
\newcommand{\bXN}{\mathbf{X}_N}

\newcommand{\bX}{\mathbf{X}}

\newcommand{\Di}{\mathsf{Discr}}
\newcommand{\Points}{\mathsf{Pts}}

\newcommand{\feta}{\mathsf{f}_{\eta}}
\newcommand{\HH}{\mathfrak{h}}
\newcommand{\nHH}{\nabla \HH}

\newcommand{\KNbeta}{\mathrm{K}^{\beta}_{N}}
\newcommand{\PNbeta}{\mathbb{P}^{\beta}_{N}}

\newcommand{\E}{\mathbb{E}}
\newcommand{\EN}{\mathbb{E}^{\beta}_N}

\renewcommand{\Pr}{\mathcal{P}}

\newcommand{\F}{\mathsf{F}}

\newcommand{\FN}{\F_N}

\newcommand{\Cc}{\mathtt{C}}
\newcommand{\kk}{\mathsf{k}}

\renewcommand{\O}{\mathcal{O}}

\newcommand{\vv}{\mathsf{v}}
\newcommand{\vu}{\vec{v}}

\newcommand{\bLa}{\bar{\La}}

\newcommand{\FbL}{\mathsf{F}_{\bLa}}
\newcommand{\FL}{\mathsf{F}_{\Lambda}}

\newcommand{\KLbeta}{\mathrm{K}^{\beta}_{\Lambda}}
\newcommand{\KLbetap}{\mathrm{K}^{\beta,p}_{\Lambda}}

\newcommand{\x}{\mathsf{x}}

\newcommand{\PpNx}{\mathsf{P}^{\beta}_{N,\x}}

\newcommand{\PI}{\mathsf{P}^{\beta}_{\infty}}
\newcommand{\bXNx}{\mathbf{X}_{N,\x}}

\newcommand{\Pp}{\mathsf{P}}

\newcommand{\fL}{f_\Lambda}
\newcommand{\fNL}{f_{\Lambda, N}}
\newcommand{\fLp}{f_\Lambda^p}
\newcommand{\normf}{\Vert f \Vert_\infty}
\newcommand{\Bin}{\mathsf{Bin}}

\newcommand{\La}{\Lambda}
\newcommand{\ind}{\mathtt{1}}

\newcommand{\Leb}{\mathrm{Leb}}

\renewcommand{\epsilon}{\varepsilon}
\newcommand{\ff}{\mathsf{f}}
\newcommand{\veta}{\vec{\eta}}

\newcommand{\Ka}{\mathcal{K}}
\newcommand{\rr}{\mathsf{r}}

\newcommand{\nabxy}{\partial^2_{12}}
\newcommand{\Error}{\mathrm{Error}}
\newcommand{\EnerPts}{\mathsf{EnerPts}}
\newcommand{\Ani}{\mathsf{Ani}}

\newcommand{\pirk}{\frac{1}{(\pi r^2)^k}}
\begin{document}
\author{\large{Thomas Leblé\thanks{ Université de Paris-Cité, CNRS, MAP5 UMR 8145, F-75006 Paris, France \texttt{thomas.leble@math.cnrs.fr}}}}
\title{\LARGE{DLR equations, number-rigidity and translation-invariance for infinite-volume limit points of the 2DOCP}}
\date{\small{\today}}

\maketitle

\begin{abstract}
We prove that, at arbitrary positive temperature, every infinite-volume local limit point of the two-dimensional one-component plasma (2DOCP, also known as Coulomb or log-gas, or jellium) satisfies a system of Dobrushin-Lanford-Ruelle (DLR) equations - in particular, we explain how to rigorously make sense of those despite the long-range  interaction.

We also show number-rigidity and translation-invariance of the limiting processes. This extends results known for the infinite Ginibre ensemble. 

The proofs combine recent results on finite 2DOCP's and classical infinite-volume techniques. 
\end{abstract}

\section{Introduction}
Let $N \geq 1$ be an integer, let $\LN \subset \R^2$ be the disk of center $0$ and radius $\RN := \sqrt{\frac{N}{\pi}}$. Denote by $\XN := (x_1, \dots, x_N)$ a $N$-tuple of points in $\LN$ and let $\bXN := \sum_{i=1}^N \delta_{x_i}$ be the associated atomic measure. We define the \textit{logarithmic interaction energy} $\FN(\bXN)$ as:
\begin{equation}
\label{def:FN}
\FN(\bXN) := \hal \iint_{(x,y) \in \LN \times \LN, x \neq y} - \log |x-y| \dd \left(\bXN- \Leb\right)(x) \dd \left(\bXN- \Leb\right)(y),
\end{equation}
where $\Leb$ denotes the  Lebesgue measure on $\R^2$. We then define a probability density on the space of $N$-tuples of points in $\LN$ by setting, for $\beta > 0$:
\begin{equation}
\label{def:PNbeta}
\dd \PNbeta(\XN) := \frac{1}{\KNbeta} \exp\left( - \beta \FN(\bXN)\right) \dd \XN,
\end{equation}
where $\KNbeta$ is the \emph{partition function}, namely the normalizing constant:
\begin{equation}
\label{def:KNbeta}
\KNbeta := \int_{\LN \times \dots \times \LN}  \exp\left( - \beta \FN(\bXN) \right)  \dd \XN.
\end{equation}
Here and below, $\dd \XN$ denotes the product Lebesgue measure $\dd x_1 \dots \dd x_N$ on $(\R^2)^N$, which serves as the reference measure. The probability measure $\PNbeta$ is called the \textit{canonical Gibbs measure of the two-dimensional one-component plasma} (2DOCP) at \textit{inverse temperature} $\beta$. We denote the expectation under $\PNbeta$ by $\EN$.

For $\x$ in $\LN$, and if $\XN$ is sampled from $\PNbeta$, define $\bXNx$ (the local point configuration \emph{seen from $\x$}) as:
\begin{equation}
\label{def:bXNx}
\bXNx := \sum_{i=1}^N \delta_{x_i -\x}.
\end{equation}
In \cite[Corollary 1.1]{Armstrong_2021}, it is proven that for all (possibly $N$-dependent) choices of $\x$ \emph{in the bulk} of $\LN$, the law $\PpNx$ of $\bXNx$ converges as $N \to \infty$ \emph{up to extraction of a subsequence} to the law $\PI$ of some random infinite point configuration (we recall in Section \ref{sec:RecallAS} what is meant by “the bulk of $\LN$” and the sense in which this convergence holds). Our results consist in a description of those limit points.
\newcommand{\bXNz}{\bX_{N, \mathsf{0}}}
\subsubsection*{A quick introduction to the DLR framework.}
Note that, focusing on point \emph{configurations} instead of $N$-tuples, we can see the Gibbs measure of the 2DOCP as the law of $\bXN = \bXNz$ with density:
\begin{equation*}
\dd \Pp_{N,\mathsf{0}}^\beta(\bXNz) := \frac{1}{\KNbeta} \exp\left( - \beta \FN(\bXNz)\right) \dd \Bin_{\LN,N}(\bXNz),
\end{equation*}
where $\Bin$ is the reference binomial point process, sampling $N$ i.i.d. uniform points in $\LN$.

It is natural to wonder whether processes $\Pp$ arising in the large-$N$ limit can be described as Gibbs measures in their own right, with densities of the form (cf. \eqref{def:PNbeta}):
\begin{equation*}
\dd \Pp(\bX) \overset{??}{=} \frac{1}{\mathrm{Normalization}} \exp\left( - \beta \mathrm{Interaction\ Energy \ of \ } \bX \right) \dd \mathrm{ReferenceMeasure}(\bX).
\end{equation*}
For several reasons - among other things, because the total interaction energy of an infinite system is usually infinite - this is in general impossible to do. The correct framework, known as Dobrushin-Lanford-Ruelle (DLR) equations, consists in describing the local, conditional laws of $\Pp$ - i.e. the law, under $\Pp$, for each bounded set $\La$, of the local configuration $\bX_\La$ in $\La$ given the configuration $\bX_{\bLa}$ \emph{outside} $\La$ - as a bona fide finite-volume Gibbs measure. We refer to \cite[Chap.~6]{MR3752129} for a pedagogical presentation, or to \cite{MR2807681} for a thorough exposition of this theory. Heuristically speaking, proving DLR equations means showing that for all bounded Borel subsets $\La \subset \R^2$, using $\bLa$ to denote the complement $\bLa := \R^2 \setminus \La$,
\begin{multline}
\label{DLR_informal}
\text{The local conditional density } \dd \Pp(\bX_\La | \bX_{\bLa}) \\
\overset{?}{=} \frac{1}{\mathrm{Normalization}} \exp\left( - \beta \mathrm{Interaction \ of \ } \bX_\La \text{ with itself \emph{and} $\bX_{\bLa}$}  \right) \dd \mathrm{ReferenceMeasure}(\bX_\La).
\end{multline}
Obtaining \eqref{DLR_informal} requires first to give a proper meaning to the right-hand side, and in particular to the interaction of a given finite configuration $\bX_\La$, living in some bounded set $\La$, with the \emph{infinite} configuration $\bX_{\bLa}$ lying outside $\La$. For finite- or short-range interactions, this is usually not an issue, and the theory is well-developed in those cases. However, \emph{the Coulomb (here log-) interaction is long-range}, indeed:
\begin{equation*}
\int_{|x| \geq 1} \log |x| \dd x = + \infty, \quad \int_{|x| \geq 1} \| \nabla \log |x| \| \dd x = + \infty, \text{ and even } \int_{|x| \geq 1} \| \mathrm{Hess} \log |x| \| \dd x = + \infty,
\end{equation*}
which poses serious difficulties, both analytical and conceptual. As a matter of fact, the mere definition of DLR equations for the 2DOCP remained an open problem (see e.g. \cite[Sec. 12]{ghosh2017fluctuations}).

\subsection{Main results}
For $N \geq 2$, let $\x = \x(N) \in \LN$ be a sequence of zooming/centering points “in the bulk”, in the sense that:
\begin{equation}
\label{def:bulk}
\liminf_{N \to \infty} \frac{1}{\sqrt{N}} \dist(\x(N), \partial \LN) > 0,
\end{equation}
and \emph{let $\PI$ be any limit point for the law $\PpNx$ of $\bXNx$}, the process seen from $\x$ as defined in \eqref{def:bXNx}.
\setcounter{theorem}{-1}
\newcommand{\Lac}{\bar{\La}}
\newcommand{\FLa}{\mathsf{F}_\La}
\newcommand{\MLa}{\mathsf{M}_\La}

\paragraph{Existence of move functions - definition of DLR equations.}
If $\La \subset \R^2$ is a bounded subset, and $\bX'$ is a finite point configuration in $\La$, we define the \emph{local interaction energy} of $\bX'$ with itself (in $\La$) as the natural extension of \eqref{def:FN}, namely:
\begin{equation}
\label{def:FLa}
\FLa(\bX') := \hal \iint_{(x,y) \in \La \times \La, x \neq y} -\log|x-y| \dd \left( \bX' -  \Leb\right)(x) \dd \left(\bX' - \Leb \right)(y),
\end{equation}
which is well-defined and takes values in $(-\infty, + \infty]$.

In order to make sense of DLR equations, we establish the following theorem, which gives a meaning to the interaction of $\bX'$ (in $\La$) with an exterior configuration $\bX$ (in $\bLa$).
\begin{theorem}
\label{theo:zero}
Let $\chi = (\chi_i)_{i \geq 0}$ be some fixed smooth partition of unity in $\R^2$ living on dyadic scales (see Section \ref{sec:notation} for details).

For all bounded Borel subset $\La \subset \R^2$, for $\PI$-a.e. $\bX$ and for all configuration $\bX'$ in $\La$ \emph{such that $\bX'$ has the same number of points as $\bX$ in $\La$}, the following quantity
\begin{equation}
\label{eq:MLa}
\MLa(\bX', \bX) := \iint_{\La \times \Lac} -\log|x-y| \dd  \left(\bX' - \bX_\La \right)(x) \dd \left( \bX_{\Lac} - \Leb \right)(y)
\end{equation}
exists and is finite in the sense that:
\begin{equation}
\label{eq:MLa2}
\iint_{\La \times \Lac} -\log|x-y| \dd  \left(\bX' - \bX_\La \right)(x) \sum_{i=0}^p \chi_i(y) \dd \left(\bX_{\Lac} - \Leb \right)(y)
\end{equation}
converges as $p \to \infty$ to a finite quantity, which we denote by $\MLa(\bX', \bX)$.
\end{theorem}
We refer to $\MLa$ as the “total move function” (following the terminology of \cite{dereudre2021dlr}), as it can be thought of as the energy cost corresponding to moving the points from $\bX$ to $\bX'$ in $\La$. Note that, for a reason that will appear later, the first integration in \eqref{eq:MLa}, \eqref{eq:MLa2} is done against the signed measure $\bX' - \bX_\La$ (which by assumption is “neutral” because $\bX'$ and $\bX$ are supposed to have the same number of points in $\La$) instead of the signed measure $\bX' - \Leb$ as e.g. in \eqref{def:FLa}.

\textbf{The DLR equations for the 2DOCP.} We define $\Conf$ as the set of (locally finite) point configurations on $\R^2$ (see Section \ref{sec:pointconfig} for details). The DLR equations for the 2DOCP interaction are then phrased as follows: we say that a point process $\Pp$ satisfies \eqref{DLR} when: 
\begin{multline}
\tag{DLR}
\label{DLR}
\text{For all bounded Borel $\La \subset \R^2$ and for all bounded function $f$ on $\Conf$,} \\
\E_{\Pp}[f] = \E_{\Pp} \left[\frac{1}{\KLbeta(\bX)} \int_{\bX' \in \Conf(\La, \Points(\bX, \La))} f\left(\bX' \cup \bX_{\bLa} \right) e^{- \beta\left( \FLa(\bX') + \MLa(\bX', \bX)  \right) } \dd \Bin_{\La, \Points(\bX,\La)}(\bX') \right],
\end{multline}
where:
\begin{itemize}
    \item $\bX$ is the random variable sampled under $\Pp$, $\Points(\bX, \La)$ is its number of points in $\La$.
    \item $\Conf(\La, n)$ is the set of point configurations in $\La$ which have exactly $n$ points (here we choose $n = \Points(\bX, \La)$).
    \item $\bX' \cup \bX_{\bLa}$ denotes the point configuration obtained by combining the points of $\bX'$ in $\La$ and the points of $\bX$ in $\bLa$.
    \item $\FLa$ (the $\La$-$\La$ interaction) is as in \eqref{def:FLa} and $\MLa$ (the $\La$-$\bLa$ interaction) is as in \eqref{eq:MLa}.
    \item $\Bin_{\La, n}$ is the \emph{binomial point process}, sampling $n$ i.i.d. uniformly distributed points in $\La$ (here we choose $n = \Points(\bX, \La)$).
    \item $\KLbeta(\bX)$ is the relevant partition function, namely:
\begin{equation*}
\KLbeta(\bX) := \int_{\bX' \in \Conf(\La, \Points(\bX, \La))} e^{- \beta\left( \FLa(\bX') + \MLa(\bX', \bX)  \right) } \dd \Bin_{\La, \Points(\bX, \La)}(\bX') .
\end{equation*}
\end{itemize}
Point processes that satisfy \eqref{DLR} are called \emph{infinite Gibbs measures} or \emph{Gibbs states} (relative to the “specification” i.e. the data of $\FLa, \MLa$... see \cite{MR2807681}).

\paragraph{Proving DLR equations for the limit points.}
The main result of this paper is the fact that weak limit points of the 2DOCP are solutions to the DLR equations.
\begin{theorem}
\label{theo:DLR}
$\PI$ satisfies \eqref{DLR}.
\end{theorem}
This gives a “physical” meaning to these processes, whose existence was obtained in \cite{Armstrong_2021} by an abstract compactness argument. It also opens the possibility to use techniques related to DLR equations in order to study those processes, see for example Theorem \ref{theo:TI}.

\paragraph{Number-rigidity.}
Our \eqref{DLR} equations are slightly unusual because we implicitly impose that the number of points inside $\La$ be determined by the outside configuration $\bX_{\bLa}$. One speaks of “canonical” equations, in the sense of the “canonical ensemble” of statistical mechanics, by opposition to “grand-canonical” (which is the usual setting for DLR equations) see e.g. \cite{georgii1976canonical}. Another way to phrase this is to observe that \eqref{DLR} equations involve a binomial point process (with a \emph{fixed} number of points) instead of a Poisson point process, which is more common.

Number-rigidity (which, under this name, has been introduced recently in \cite{Ghosh_2017}) is the property of a point process such that for every bounded Borel set $\La$, the number of points \emph{inside} $\La$ is a.s. equal to a certain function of the configuration \emph{outside} $\La$. 

Obtaining canonical DLR equations is very natural for Gibbs measures that are number-rigid. And as a matter of fact, we prove:
\begin{theorem}
\label{theo:NR}
$\PI$ is \emph{number-rigid}. 
\end{theorem}
This result is new for $\beta \neq 2$. It provides a natural one-parameter family of two-dimensional number-rigid point processes which (conjecturally) interpolate between a Poisson point process ($\beta \to 0$) and a lattice ($\beta \to \infty$).

\paragraph{Translation-invariance.}
Our final result belongs to a line of questions about continuous symmetry breaking (or the absence thereof) in two-dimensional systems. 

In Section \ref{sec:good} we introduce a set of “good” point processes that enjoy the same good properties as the weak limit points (see Section \ref{sec:prelimINF}), so that in particular the functions $\MLa$ (see \eqref{eq:MLa}) are well-defined and \eqref{DLR} makes sense. We also include in the definition of “good” processes a couple of technical requirements. We then prove:
\begin{theorem}
\label{theo:TI}
If $\Pp$ is a “good” point process which is solution of \eqref{DLR}, then $\Pp$ is \emph{translation-invariant\footnote{$\Pp$ is translation-invariant when for all $u \in \R^2$ the push-forward of $\Pp$ by $\bX \mapsto \sum_{x \in \bX} \delta_{x+u}$ is equal to $\Pp$.}}.
\end{theorem}
By design, all weak limit points are “good” processes, and thus, combining Theorem \ref{theo:DLR} and Theorem \ref{theo:TI}:
\begin{corollary}
$\PI$ is translation-invariant.
\end{corollary}
The fact that this statement holds for $\beta$ arbitrarily large  is physically significant. Indeed, it is expected (see e.g. \cite[Sec. VI]{lewin2022coulomb} for a survey) that around $\beta \approx 140$, the system becomes a “crystal”, and that as $\beta \to \infty$ it becomes a lattice. Theorem \ref{theo:TI} shows that even in its hypothetical “crystal” phase, and despite the long-range nature of the interaction, the 2DOCP does not break translation-invariance.

\subsection{The 2DOCP}
The 2DOCP (or Coulomb gas, or log-gas) is a well-studied model of statistical physics - we refer to the recent survey \cite[Section III]{lewin2022coulomb} for an overview and many references.  Its finite-volume properties have been the object of many recent studies in the mathematical literature (see the monograph \cite{serfaty2024lectures}). 

The interaction energy \eqref{def:FN} should be thought of as the logarithmic/Coulomb pairwise interaction of a system made of $N$ point charges placed at $x_1, \dots, x_N$ together with a “uniform neutralizing background” (to quote from the physics literature) of density $-1$ on $\LN$.

Some important physics papers about this model (for example \cite{jancovici1993large}) consider “infinitely extended” 2DOCP's, however their mere existence was until recently unclear from a mathematical perspective. The present paper goes relies on and goes beyond the abstract existence result of \cite{Armstrong_2021} (which consists in a compactness argument) by giving an interpretation of the infinite-volume system in line with the classical theory of infinite Gibbs measures used in mathematical statistical mechanics.

\paragraph{Hard versus soft confinement.} We have chosen here to work with a 2DOCP whose particles are forced to live in a given “box”, namely the disk $\LN$, as this is the usual choice in the physics literature (though that box might sometimes be a square). Another possibility consists in lifting the constraint $\XN \in (\LN)^N$ and allowing for any $\XN \in (\R^2)^N$, while adding to the interaction energy a term of the form $\sum_{i=1}^N \zeta_N(x_i)$, where the “effective confinement” $\zeta_N$ vanishes on $\LN$ and grows fast enough at infinity in order to force “most” of the particles to be “inside or close to” $\LN$ (see \cite[Sec. 5.1]{serfaty2024lectures}).

In the first case (strict confinement), there is a “hard edge along $\partial \LN$, while in the second case (soft confinement), inspired by models of Random Matrix Theory, some outliers may live outside $\LN$. As a rule of thumb, when studying phenomena that take place “in the bulk” (away from $\partial \LN$), this choice does not matter very much. In particular, we could run most of the present analysis in the case of a soft confinement with no changes. There is however one technical point in the proof of Theorem \ref{theo:DLR} (DLR equations) which seems to be more delicate with a hard edge (i.e. in the setting of the present paper), we discuss this in Section \ref{sec:LastLayer}.

\paragraph{The Ginibre ensemble.} The value $\beta = 2$ of the inverse temperature parameter stands out because the system is then “integrable” and many explicit computations become tractable. In particular, for $\beta = 2$:
\begin{enumerate}
    \item Existence of an infinite-volume limit (and not only of limit \emph{points}) for the local point processes $\PpNx$'s is known since the work of Ginibre \cite{ginibre1965statistical}. It is called the Ginibre point process, infinite Ginibre ensemble, Ginibre random point field... Moreover:
\begin{itemize}
    \item The limit does not depend on $\x$ as long as it is chosen in the bulk of $\LN$. 
    \item The limit process is determinantal (see \cite{hough2006determinantal}). 
    \item The limit process is translation-invariant (which is not obvious).
\end{itemize}
    \item Number-rigidity has been proven in \cite{Ghosh_2017} among the first examples of rigid point processes.
    \item DLR equations have been established in \cite{bufetov2017conditional}.
\end{enumerate}
\begin{remark}
Strictly speaking, the Ginibre ensemble correspond to a 2DOCP at inverse temperature $\beta = 2$ \emph{and with a certain “soft” confinement}. The “hard edge” case has however also been studied for a long time, see \cite{jancovici1982classical} for an early occurence in the physics literature and \cite{ameur2024random} (and the references therein) for a recent mathematical treatment.
\end{remark}

\paragraph{Scaling convention.} In this paper, the relevant lengthscales vary from order $1$ (inter-particle distance) to order $\sqrt{N}$ (diameter of the finite-$N$ system). This is appropriate in order to take $N \to \infty$ and study the infinite-volume limit. Some papers choose the “random matrix” convention and scale everything down by $\sqrt{N}$, so the inter-particle distance is $\sim N^{-\hal}$ and the global lengthscale is $\sim 1$. This means that some back-and-forth is sometimes required when importing results.

\subsection{Connection with the literature}
\paragraph{The one-dimensional log-gas.} The one-dimensional log-gas is a closely related model, whose definition looks very similar to \eqref{def:FN}, \eqref{def:PNbeta} except that the $N$ particles are now constrained to live on the real line (either strictly confined within a line segment of length $N$ or with some “soft” confinement). This model (which is also  called a “$\beta$-ensemble”) has received a lot of attention, in part because of its link with Random Matrix Theory, as first observed by Dyson \cite{dyson1962statistical}. The value $\beta = 2$ is again remarkable and has been studied thoroughly, as well (to a lesser extent) as the cases $\beta = 1, 4$. Recent works \cite{valko2009continuum} (see also \cite{killip2009eigenvalue}) building upon the celebrated “tridiagonal model” of \cite{dumitriu2002matrix} gave some understanding of the general $\beta > 0$ case regarding infinite-volume limits. 

In summary:
\begin{itemize}
     \item For $\beta = 1, 2, 4$, limiting processes exist and have a remarkable algebraic structure (determinantal for $\beta = 2$, Pfaffian for $\beta = 1, 4$). For $\beta = 2$, number-rigidity follows from \cite{bufetov2017quasi} and DLR equations were proven in \cite{bufetov2020conditional} (see also \cite{kuijlaars2019universality}). For $\beta = 1, 4$, number-rigidity is proven in \cite{bufetov2019number}. Those processes are known to be translation-invariant. We are not aware of ad hoc DLR equations.
     \item For $\beta > 0$ arbitrary, by \cite{valko2009continuum}:
     \begin{enumerate}
         \item There is existence of an infinite-volume limit (and not only of limit \emph{points}) for $\PpNx$. It is named the Sine-$\beta$ point process.
         \item The limit does not depend on $\x$, up to a density scaling, as long as $\x$ is chosen in the bulk of the system.
         \item The Sine-$\beta$ process is translation-invariant (which is not obvious).
     \end{enumerate}
     \item Moreover, closely related to the results of the present paper, it was recently proven that:
     \begin{enumerate}
         \item Sine-$\beta$ is number-rigid for all $\beta > 0$ (\cite{chhaibi2018rigidity}, \cite{dereudre2021dlr}).
         \item DLR equations hold \cite{dereudre2021dlr}. Those equations have then been put to use in \cite{leble2021clt} in order to study the fluctuations of linear statistics for Sine-$\beta$, which shows that the “physical” description of the limiting process given by the DLR formalism can be fruitful.
     \end{enumerate}
 \end{itemize} 
The present work originates from an attempt at adapting the methods of \cite{dereudre2021dlr} to the two-dimensional log-gas.

\begin{remark}[Two- versus one-dimensional log-gas]
Giving a sense to the move functions (Theorem~\ref{theo:zero}) is a major issue when trying to obtain DLR equations for systems with such long-range interactions. A similar problem arose in \cite{dereudre2021dlr} concerning the $1d$ log-gas, but the problem of convergence was easier to solve there, due to the fact that the $\log$ potential is less long-range in dimension $1$ than in dimension $2$.

At a technical level, the difference between the two situations can be seen as follows: for the $1d$ log-gas, one could use a decomposition of the space (the real line) into line segments of length $1$ (namely, the $\chi_i$'s in Theorem \ref{theo:zero} can be chosen as indicator functions of small segments instead of a smooth partition of unity associated to dyadic annuli). Using fairly rough controls (called discrepancy estimates, and not even stated in a sharp form) on the contribution coming from each segment was then enough to deduce almost sure convergence as $p \to \infty$ in \eqref{eq:MLa2}. Here for the 2DOCP, one needs to cut the space quite differently, using smooth cut-offs living at dyadic scale, and to import fine estimates controlling the size of fluctuations for linear statistics of smooth functions. 

In that regard, the key technical input that was needed for the present work is the result of \cite{Armstrong_2021} stating local laws that are valid \emph{down to the microscopic scale}.
\end{remark}

\paragraph{Empirical field and free energy minimizers.}
\newcommand{\Fbeta}{\mathcal{F}_\beta}
\newcommand{\Emp}{\mathrm{Emp}}
Some understanding about the infinite-volume behavior of the 2DOCP at microscopic/local scale can be found in \cite{MR3735628} thanks to a large deviation principle for a microscopic observable $\Emp_N$ called the \emph{empirical field}. In short, $\Emp_N$ is obtained by considering the local point configuration $\bXNx$ (as introduced above in \eqref{def:bXNx}) and adding another layer of randomness by \emph{choosing $\x$ uniformly at random within (the bulk of) $\LN$}. This has the effect of “averaging” over the values of $\x$ and thus only speaks about the \emph{average} local behavior. It is proven in \cite{MR3735628} that $\Emp_N$ concentrates as $N \to \infty$ around minimizers of a certain free energy functional $\Fbeta$ defined on the space of translation-invariant point processes.

It is important to keep in mind that a priori, the limit (points) of $\Emp_N$ and the limit (points) of $\PpNx$ \emph{might not coincide}, unless e.g. all the $\PpNx$'s have the same limit (as is the case for $\beta =2$). For example, limit points of $\Emp_N$ are always translation-invariant “for free” because of the averaging, but proving it for the “non-averaged” local process is significantly more subtle.

\paragraph{(Non)-number rigidity.} As mentioned above, related number-rigidity processes include:
\begin{itemize}
    \item The “classical” Sine-process and its Pfaffian friends (log-gas, dimension $1$, $\beta = 1, 2, 4$), by the works of Bufetov et al.
    \item The Sine-$\beta$ process (log-gas, dimension $1$, $\beta > 0$) by \cite{chhaibi2018rigidity,dereudre2021dlr}.
    \item The infinite Ginibre ensemble (log-gas, dimension $2$, $\beta =2$), by \cite{Ghosh_2017}.
\end{itemize}
Some non-rigidity results have recently been proven concerning Riesz gases, i.e. systems of particles where the logarithmic interaction $-\log|x-y|$ is replaced by a power-law $\frac{1}{|x-y|^s}$ (we refer again to the survey \cite{lewin2022coulomb}). Among those:
\begin{itemize}
    \item For $d \geq 2$ and $s \in (d-1, d)$, the existence of non-rigid Gibbs states at arbitrary temperature has been found in \cite{dereudre2023number}, using the DLR framework. 
    \item Non-rigidity of higher-dimensional Coulomb gases (dimension $d \geq 3$, $s = d-2$, $\beta > 0$) is proven in \cite{thoma2023non} with a new, DLR-free argument.
\end{itemize}

It has recently been observed that number-rigidity of a point process plays a central role in certain properties of the dynamics associated to this process, see \cite{osada2024vanishing} (“dynamical rigidity” for the Ginibre ensemble) and \cite{suzuki2024ergodicity} (ergodicity for Ginibre and other rigid point processes). One can expect that our proof of number-rigidity for the 2DOCP at all $\beta > 0$ will have consequences for the study of the associated dynamics. Note that these dynamics remain to be defined - in fact, we suspect that our Theorem \ref{theo:zero} might be helpful in that regard, as it gives a way to make sense of the interaction of a particle with an infinite system.

\paragraph{Translation-invariance.}
Note that from the mere definition \eqref{def:bXNx} of $\bXNx$ (choose e.g. $\x = 0$ for all $N$) it is not obvious at all that the limiting process should have a translation-invariant distribution. Indeed, although the interaction $-\log|x-y|$ depends only on the relative position of $x$ and $y$, the finite-volume Gibbs measure $\PNbeta$ itself cannot be translation-invariant because of the finite domain (this would of course be different if one had chosen periodic boundary conditions or averaged over choices of $\x$ by considering empirical fields). 

As put in \cite{MR1886241}: \textit{“This leads us to the question of whether translation invariance necessarily holds, i.e., under what conditions all Gibbs measures are translation invariant. (...) For point particle systems, translation is a continuous symmetry. So our question leads us into the realm of Mermin-Wagner\footnote{We emphasize that the Mermin-Wagner theorem is... a theorem, with certain assumptions, and that continuous symmetries \emph{can} (quite obviously) be broken if the interactions are too strong.} resp. Dobrushin-Shlosman theory of conservation of continuous symmetries in two-dimensional systems.”} There are general, powerful results due to \cite{frohlich1981absence,frohlich1986absence} and \cite{MR1886241}, extended e.g. in \cite{richthammer2005two,richthammer2007translation,richthammer2009translation}, showing that under certain conditions all Gibbs states (i.e. all solutions to the DLR equations) must be translation-invariant. However the existing assumptions on regularity and decay of the interaction potential do not apply to the Coulomb setting. We show that those methods can nonetheless be used (in combination with recent results on 2DOCP's) to prove Theorem \ref{theo:TI}. 

\begin{remark}
There exists a finite-volume approximate translation-invariance estimate for the 2DOCP, which appears as a key ingredient in \cite{leble2021two}. This implies translation-invariance of the infinite-volume limit points. Although it relies on similar ideas, it is fair to say that the argument in \cite{leble2021two} is significantly more involved than Theorem \ref{theo:TI}, due to a puzzling simplification occurring in infinite volume. We choose here to present the simpler proof obtained by directly working with infinite systems, which illustrates both the benefit of having DLR equations and the robustness of classical infinite-volume arguments.
\end{remark}

\subsection{Open questions}
\label{sec:open}
Let us mention three natural important open questions:
\begin{itemize}
\item Does $\PpNx$ admit a limit (and not only limit points) as $N \to \infty$? 
\item Does \eqref{DLR} admit a unique solution?
\item Do solutions of \eqref{DLR} and minimizers of the free energy $\Fbeta$ of \cite{MR3735628} coincide?
\end{itemize}
Of course, in view of our results, a positive answer to the second question would also bring a positive answer to the first one. A positive answer for $\beta$ \emph{small enough} would already be very interesting - there are general results hinting at it, like Dobrushin's uniqueness criterion (see \cite[Chap. 8]{MR2807681}) for lattice systems with short-range interactions, but those are not proven for long-range continuous systems.

\paragraph{Edge behavior.}
The result of \cite[Cor. 1.3]{thoma2022overcrowding}, extending \cite[Cor. 1.1]{Armstrong_2021}, shows (at least in the case of a soft confinement) the existence of limit points for $\PpNx$ even if $\x$ is near or on the boundary of $\LN$. Understanding the limit points near the boundary, proving DLR equations etc. would be an interesting task. Note that for the Ginibre ensemble ($\beta =2$), the edge process is known (\cite{forrester1999exact}), but we are not aware of any DLR equations for it.


\paragraph{Riesz gases.} It would be natural to seek to derive DLR equations for the entire family of Riesz gases in dimension $d \geq 1$, with $s \in [d-2, d)$ - and in particular for all Coulomb gases in dimension $d \geq 3$. However, the proof presented here relies on a fine understanding of the 2DOCP (local laws, good controls of linear statistics and of discrepancies) that are not available for the full Riesz family yet.

\subsection{Plan of the paper}
\begin{itemize}
    \item We gather some definitions in Section \ref{sec:preliminaries}. 
    \item In Section \ref{sec:ResFinite}, we recall recent results on the finite-volume 2DOCP.
    \item In Section \ref{sec:prelimINF} we show that the convergence result of \cite{Armstrong_2021} can be upgraded to a much stronger topology. We use this to pass finite-volume results to the infinite-volume limit.
    \item We prove number-rigidity (Theorem \ref{theo:NR}) in Section \ref{sec:numberRigProof}. It does not rely on DLR equations.
    \item In Section \ref{sec:ProofMaking} we prove that \eqref{DLR} equations \emph{make sense} by proving Theorem \ref{theo:zero}.
    \item In Section \ref{sec:DLR_expression} we prove Theorem \ref{theo:DLR}: weak limit points of the 2DOCP are solutions to \eqref{DLR}.
    \item In Section \ref{sec:ProofTI} we show the translation-invariance result stated in Theorem \ref{theo:TI}.
    \item Finally, an appendix contains the proof of some auxiliary results.
\end{itemize}

\subsubsection*{Acknowledgments.}
The author acknowledges the support of JCJC grant ANR-21-CE40-0009 from Agence Nationale de la Recherche. We thank Thibaut Vasseur for insightful comments.

\section{Preliminaries}
\label{sec:preliminaries}

\subsection{Notation}
\label{sec:notation}

\begin{itemize}
    \item We let $\DD(x,r)$ be the disk of center $x$ and radius $r$. We simply write $\DD(r)$ for $\DD(0,r)$.
    \item We denote indicator functions by $\ind$.
    \item We denote “universal” constants by $\Cc$, constants depending only on $\beta$ by $\Cc_\beta$, and indicate the dependencies otherwise. Those constants are always chosen “as large as needed” and may change from line to line.
\end{itemize}

\paragraph{Notation for the size of derivatives.} For $\varphi : \R^2 \to \R$, we write:
    \begin{equation}
    \label{notDeriv}
|\varphi|_\0 := \sup_{x \in \R^2} |\varphi(x)|, \quad |\varphi|_\1 := \sup_{x \in \R^2} \|\nabla \varphi(x)\|, \quad |\varphi|_\2 := \sup_{x \in \R^2} \|\mathrm{Hess} \varphi(x)\| \text{ etc.}
    \end{equation}
    We also use $|\varphi|_\1$ to denote the Lipschitz semi-norm of $\varphi$.

If $\Ka : \R^2 \times \R^2 \to \R$ is a function of two variables $x = (x_1, x_2)$ and $y = (y_1, y_2)$, we denote by $\nabxy \Ka(x,y)$ the matrix:
    \begin{equation}
    \label{def:nabxy}
        \nabxy \Ka(x,y) := 
            \begin{pmatrix}
                \partial_{x_1 y_1} \Ka(x,y) & \partial_{x_1 y_2} \Ka(x,y) \\
                \partial_{x_2 y_1} \Ka(x,y) & \partial_{x_2 y_2} \Ka(x,y)
            \end{pmatrix}.
    \end{equation}
    (note that these are all second derivatives with respect to $x$ \emph{and} $y$)
and we let $|\Ka|_{\1 + \1}$ be the supremum of the mixed second derivatives of $\Ka$:
    \begin{equation}
    \label{Ka1p1}
    |\Ka|_{\1 + \1} := \sup_{x,y \in \R^2 \times \R^2} \| \nabxy \Ka(x,y) \|. 
    \end{equation}

\paragraph{A dyadic partition of unity.}
Throughout the paper, we fix a family $\chi := (\chi_i)_{i \geq 0}$ of smooth non-negative functions on $\R^2$ such that:
\begin{itemize}
    \item The $\chi_i$'s form a partition of unity, i.e. $\sum_{i \geq 0} \chi_i(x) = 1$ for all $x \in \R^2$.
    \item $\chi_0$ is supported on $\DD(2)$ and for all $i \geq 1$, $\chi_i$ is supported on the dyadic annulus $\DD(2^{i+1}) \setminus \DD(2^{i-1}) $.
    \item There exists a constant $\Cc_\chi$ such that for all $i \geq 0$ and with the notation of \eqref{notDeriv}:
    \begin{equation}
    \label{assumchii}
    |\chi_i|_{\kk} \leq \Cc_\chi 2^{-i \kk}, \ \kk = 0, 1, 2, 3.
    \end{equation}
\end{itemize}

\subsection{Point configurations and point processes}
\label{sec:pointconfig}
\paragraph{Generalities.}
We denote by $\Conf$ the space of locally finite point configurations in $\R^2$. We usually denote such configurations with a bold typeface and we see them either as a collection of points or as a sum of Dirac masses $\bX = \sum_{x \in \bX} \delta_x$.

If $\bX$ is a point configuration and $\La \subset \R^2$ is Borel measurable, we denote by $\Points(\bX, \La)$ the number of points of $\bX$ in $\La$.

The space $\Conf$ is endowed with the \emph{vague topology} (the topology of “$w^{\#}$-convergence” in \cite[Sec.~9.1]{daley2007introduction}), namely the coarsest topology that makes the linear statistics $\bX \mapsto \int f(x) \dd \bX(x)$ continuous for all test functions $f : \R^2 \to \R$ that are \emph{continuous and compactly supported}. It is well-known (see e.g. \cite[Proposition 9.1.IV. i)]{daley2007introduction}) that $\Conf$ endowed with this topology is Polish.

The Borel $\sigma$-algebra on $\Conf$ coincides with the one generated by all the counting functions $\bX \mapsto \Points(\bX, \La)$, where $\La$ is any bounded Borel subset of $\R^2$ (see \cite[Proposition 9.1.IV. ii)]{daley2007introduction}).

The space of (laws of) \emph{point processes} $\Pr(\Conf)$ is the set of probability measures on $\Conf$. We will consider several topologies on $\Pr(\Conf)$. The most natural for the moment is the usual “weak topology” or topology of \emph{weak convergence of measures}, for which a sequence $(\Pp_n)_{n \geq 0}$ converges to $\Pp$ if and only if
\begin{equation}
\label{WeakA}
\lim_{n \to \infty} \E_{\Pp_n}[f] = \E_{\Pp}[f] \text{ for all } f : \Conf \to \R \text{ bounded and continuous for the vague topology.}   
\end{equation}

\paragraph{Local functions and events.}
We say that a function $f : \Conf \to \R$ is \emph{local} when there exists a bounded Borel set $\La$ such that $f(\bX) = f(\bX')$ whenever $\bX = \bX'$ on $\La$. We say that an event $A \subset \Conf$ is local when $\1_A$ is local. It is worth emphasizing that local events generate the Borel $\sigma$-algebra on $\Conf$.

\paragraph{Discrepancies and fluctuations.}
If $\La$ is bounded and $\bX \in \Conf$, we write $\Di(\bX, \La)$ for the difference 
\begin{equation*}
\Di(\bX, \La) := \Points(\bX, \La) - \Leb(\La).
\end{equation*} 
More generally, if $\varphi : \R^2 \to \R$ is a piecewise continuous, compactly supported function, we write:
\begin{equation}
\label{def:Fluct}
\Fluct[\varphi](\bX) := \int_{\R^2} \varphi(x) \dd \left(\bX - \Leb  \right)(x),
\end{equation}
called the “fluctuation of the linear statistics” associated to $\varphi$. Note that  we retrieve $\Di$ when $\varphi = \ind_{\La}$.

\subsection{The convergence result of \texorpdfstring{\cite{Armstrong_2021}}{AS21}}
\label{sec:RecallAS}
The result of \cite[Corollary 1.1]{Armstrong_2021} is the following: if for each $N \geq 1$ we choose a point $\x = \x(N)$ in $\LN$ such that $\dist(\x, \partial \LN) \geq \Cc_\beta N^{1/4}$ (where $\Cc_\beta$ is some constant, depending only on $\beta$, defined in \cite{Armstrong_2021}), then the law $\PpNx$ of $\bXNx$ converges, as $N \to \infty$, to some point process $\PI$:
\begin{enumerate}
\item up to extraction of a subsequence 
\item in the weak topology of $\Pr(\Conf)$ (see \eqref{WeakA}).
\end{enumerate}
Their compactness argument relies on \emph{local laws}, here in the form of a bound on the moment of the number of points in large boxes, valid \emph{down to the microscopic} scale (a major improvement upon previous local laws given in \cite{leble2017local,bauerschmidt2017local} which dealt with mesoscopic scales) and the tightness of .

We do not know how to dispense with extracting a subsequence in order to prove existence of a limiting point process (see the open questions in Section \ref{sec:open}), but working with limit points is enough for our purposes. The basic idea will be (of course) to pass certain good properties of the finite-volume 2DOCP to the limit $N \to \infty$. To do so, we will need a strong enough notion of convergence - an enhancement over the weak topology used in \cite{Armstrong_2021} - which we present in Section~\ref{sec:PrecisionAS}.

\section{Results in finite volume}
\label{sec:ResFinite}
We collect here some recent results in the analysis of finite-volume 2DOCP's.

For $x$ a point in $\LN$ and $\ell > 0$ a lengthscale, we say that “$(x, \ell)$ satisfies \eqref{condi:LL}” when: 
\begin{equation}
\label{condi:LL}
\ell \geq \rho_\beta, \quad \dist(\DD(x, \ell), \partial \LN) \geq \Cc_\beta N^{1/4},
\end{equation}
where $\rho_\beta, \Cc_\beta$ are some large constants, depending only on $\beta$, introduced in \cite[(1.15) and (1.16)]{Armstrong_2021}.

Note that the definition \eqref{def:bulk} of “bulk” that we chose to state our results clearly implies that the points $\x = \x(N)$ in the sequence satisfy \eqref{condi:LL} for all $\ell \geq \rho_\beta$ fixed and $N$ large enough (depending on $\ell$ and on the sequence $\x$).

\subsubsection*{Bounds on the number of points}
\newcommand{\Ccb}{\Cc_\beta}
\begin{lemma}
There exists $\Ccb$ depending only on $\beta$ such that, if \eqref{condi:LL} is satisfied:
\label{lem:points}
\begin{equation}
\label{eq:bound_points_finite}
\log \EN \left[ \exp\left(\frac{1}{\Ccb} \Points^2(\bXN, \DD(x, \ell))\right)  \right] \leq \Ccb \ell^4, \quad \log \EN \left[ \exp\left(\frac{1}{\Ccb} \Points(\bXN, \DD(x, \ell))\right)  \right] \leq \Ccb \ell^2 .
\end{equation}
\end{lemma}
\begin{proof}[Proof of Lemma \ref{lem:points}]
From \cite[(1.19)]{Armstrong_2021} we know that, writing for short $\Di = \Di(\bXN, \DD(\x, \ell))$:
\begin{equation}
\label{AS119}
\log \EN \left[  \exp\left(\frac{1}{\Ccb} \Di^2 \min\left(1, \frac{|\Di|}{\ell^2} \right)\right) \right] \leq \Ccb \ell^2.
\end{equation}
Now, if $|\Di| \leq 10 \ell^2$ then the number of points in $\DD(\x, \ell)$ is bounded by $\Cc \ell^2$. Otherwise, there is a large excess and $\Points^2 \leq 2 \Di^2$. We thus have:
\begin{multline*}
\EN \left[ \exp\left(\frac{1}{\Ccb} \Points^2(\bXN, \DD(x, \ell))\right)  \right] \\
= \EN \left[ \exp\left(\frac{1}{\Ccb} \Points^2(\bXN, \DD(x, \ell))\right) \ind_{|\Di| \leq 10 \ell^2} \right] + \EN \left[ \exp\left(\frac{1}{\Ccb} \Points^2(\bXN, \DD(x, \ell))\right) \ind_{|\Di| > 10 \ell^2} \right] \\
\leq
 \exp\left(\frac{\Cc^2}{\Ccb} \ell^4 \right) + \EN \left[ \exp\left(\frac{2}{\Ccb} \Di^2(\bXN, \DD(x, \ell))\right) \ind_{|\Di| > 10 \ell^2} \right] \\
 \leq  \exp\left(\frac{\Cc^2}{\Ccb} \ell^4\right)  + \EN \left[ \exp\left(\frac{2}{\Ccb} \Di^2 \min\left(1, \frac{|\Di|}{\ell^2} \right)\right)  \right],
\end{multline*}
and we can bound the last term using \eqref{AS119} (up to choosing a larger constant $\Ccb$ in \eqref{eq:bound_points_finite} than in \eqref{AS119}).

This proves the first inequality, which of course implies the second one.
\end{proof}

\subsubsection*{Bounds on linear statistics: Lipschitz case}
\newcommand{\Lell}{\mathcal{L}_{x, \ell}}
\begin{lemma}
\label{lem:fluct_Lipschitz}
Let $x$ be a point in $\R^2$ and $\ell > 0$ be some lengthscale such that $(x, \ell)$ satisfies \eqref{condi:LL}. Let $\Lell$ be the set of all functions $\varphi : \R^2 \to \R$ that are $\ell^{-1}$-Lipschitz and compactly supported on the disk $\DD(x, 4 \ell)$. We have, 
\begin{equation}
\label{eq:bound_fluct_Lipschitz}
\log \EN \left[ \sup_{\varphi \in \Lell} e^{\frac{1}{\Cc_\beta} \Fluct^2[\varphi]}  \right] \leq \ell^2, \quad 
\log \EN \left[ \sup_{\varphi \in \Lell} e^{\Fluct[\varphi]}  \right] \leq \Ccb \ell.
\end{equation}
\end{lemma}
In particular we get, using Jensen's inequality,
\begin{equation}
\label{ExpStatLip}
\EN\left[ \sup_{\varphi \in \Lell}\left|\Fluct[\varphi]\right| \right] \leq \Ccb \ell.
\end{equation}
We can thus uniformly control linear statistics of Lipschitz test functions within the class $\Lell$. The first bound in \eqref{eq:bound_fluct_Lipschitz} looks like \cite[(1.20)]{Armstrong_2021} but unfortunately it is not stated there in a uniform way. The second control in \eqref{eq:bound_fluct_Lipschitz} can be found \cite[Lemma 2.4]{lambert2024law}, relying on older ideas. We recall the proof in Section \ref{sec:FluctEnergy} as we also need here various small extensions.

 The next result improves on this control for linear statistics of smoother test functions (we get $\O(1)$ instead of $\O(\ell)$), but it is now stated function-wise (see \cite[Sec. 3.4.]{lambert2024law} for a uniform bound).

\subsubsection*{Bounds on linear statistics: smooth case.}
It is a remarkable feature of the 2DOCP that fluctuations of \emph{smooth enough} linear statistics are typically of order $1$ (as proven in \cite{leble2018fluctuations,MR4063572,serfaty2020gaussian}).
\begin{lemma}
\label{lem:fluct_smooth}
Let $\varphi$ be a test function of class $C^3$, compactly supported in some disk $\DD(x, 4\ell)$ with $(x, \ell)$ satisfying \eqref{condi:LL}. Assume that $|\varphi|_{\kk} \leq \Cc_\varphi \ell^{-\kk}$ for $\kk = 1,2, 3$ for some constant $\Cc_\varphi$. Then for $|s| \leq \frac{\ell^2}{\Cc_\beta \Cc_\varphi}$ we have
\begin{equation}
\label{eq:bound_Fluct_smooth}
\left|\log \EN \left[ e^{s \Fluct[\varphi]}  \right]\right| \leq \Cc_\beta \left(|s| \Cc_\varphi + s^4 \Cc_\varphi^4 \right).
\end{equation}
\end{lemma}
\begin{proof}
This is a re-writing of \cite[Cor. 2.1]{serfaty2020gaussian} with our scaling convention.
\end{proof}
In particular we get (taking $s = \frac{\pm 1}{\Cc_\beta \Cc_\varphi}$ and applying Jensen's inequality): 
\begin{equation}
\label{ExpFluctSmooth}
\EN\left[ \left| \Fluct[\varphi] \right| \right] \leq \Cc'_\beta \Cc_\varphi.
\end{equation}

\subsubsection*{Bounds on “bilinear” statistics.}
The following statement is new as such, but the method has been used previously. Recall the “mixed derivative” notation $|\cdot|_{\1 + \1}$ from \eqref{Ka1p1}.
\newcommand{\LLLL}{\mathcal{L}^{(2)}_{L_1, L_2}}
\begin{lemma}
\label{lem:bilinear}
Let $L_1, L_2$ be two lengthscales with $\rho_\beta \leq L_1 \leq L_2$, let $\LLLL$ be the set of all functions $\Ka$ which are in $C^2(\R^2 \times \R^2)$ with compact support in $\DD(x, L_1) \times \DD(y, L_2)$ where both $(x, L_1)$ and $(y, L_2)$ satisfy \eqref{condi:LL}, and such that:
\begin{equation*}
|\Ka|_{\1 + \1} \leq (L_1 L_2)^{-1}.
\end{equation*}
Then for all $s$ such that $|s| \leq \frac{1}{\Cc_\beta} \frac{L_1}{L_2}$:
\begin{equation}
\log \EN \left[ \sup_{\Ka \in \LLLL} \exp\left( s \iint_{\LN \times \LN} \Ka(x,y) \dd \left(\bXN - \Leb  \right)(x) \dd \left(\bXN - \Leb  \right)(y) \right) \right] \leq |s| \Cc_\beta L_1 L_2.
\end{equation}
\end{lemma}
\begin{proof}
The proof is essentially the same as for Lemma \ref{lem:fluct_Lipschitz}, we present it in Section \ref{sec:FluctEnergy}.
\end{proof}

\subsubsection*{Bounds on linear statistics hitting the hard edge.}
\newcommand{\bphi}{\bar{\varphi}}
\newcommand{\LLL}{\mathcal{L}_{rot}}
\newcommand{\LLN}{\mathcal{L}_N}
\newcommand{\LLNR}{\mathcal{L}_{N, rad}}
In the setting of Lemma \ref{lem:fluct_Lipschitz}, the test functions $\varphi$ are compactly supported within the support $\LN$ of the system. However, in the decomposition into dyadic scales that we use below, the “last level” will contain the boundary of $\LN$, and we need to make sure that the conclusions are still valid.

\begin{lemma}[Lipschitz linear statistics hitting the hard edge]
\label{lem:LipschitzHardWall}
Let $\LLN$ be the set of all Lipschitz functions $\varphi : \R^2 \to \R$ with $|\varphi|_\0 \leq 1$ and $|\varphi|_\1 \leq N^{-\hal}$. We have;
\begin{equation}
\label{HardWallLipschitz}
\EN\left[ \sup_{\varphi \in \LLN} e^{\Fluct[\varphi \1_{\LN}] } \right] \leq \Cc_\beta \sqrt{N}, \quad \EN \left[ \sup_{\varphi \in \LLN} |\Fluct[\varphi \1_{\LN}]| \right] \leq \Cc_\beta \sqrt{N}.
\end{equation}
\end{lemma}
\begin{proof}
The proof is in fact the same as for Lemma \ref{lem:fluct_Lipschitz}, see Section \ref{sec:FluctEnergy}.
\end{proof}

Finally, we will need the following statement, which is new. It gives a slight but crucial improvement on the previous one when the test function is assumed to have \emph{radial symmetry}, and relies on the difficult result of \cite{leble2021two}.

\begin{lemma}[Radial linear statistics hitting the hard edge]
\label{lem:SmoothHardWall}
Let $\LLNR$ be the set of all \emph{radially symmetric} Lipschitz functions $\varphi : \R^2 \to \R$ with $|\varphi|_\0 \leq 1$ and $|\varphi|_\1 \leq N^{-\hal}$. We have, with a $o(\sqrt{N})$ depending on~$\beta$:
\begin{equation}
\label{HardWallSmooth}
\EN \left[ \sup_{\varphi \in \LLNR} |\Fluct[\varphi \1_{\LN}]|\right] = o\left(\sqrt{N}\right),
\end{equation}
\end{lemma}
We postpone the proof to Section \ref{sec:ProofHardWall}.

\subsubsection*{Bounds on correlation functions.}
\newcommand{\Cckb}{\Cc_{k, \beta}}
\newcommand{\chipr}{\chi_r}
\newcommand{\chimr}{\chi_{r/2}}
We will use the following result of \cite{thoma2022overcrowding}. For $k \geq 1$, define the $k$-point correlation function of $\PNbeta$ as:
\begin{equation}
\label{def:rho_k}
\rho_{k, \PNbeta}(y_1, \dots, y_k) = \frac{N!}{(N-k)!} \lim_{r \to 0} \frac{1}{\pirk} \EN\left(\prod_{i=1}^n \ind_{\DD(y_i, r)}(x_i) \right),
\end{equation}
or equivalently, by duality, as:
\begin{equation*}
\EN\left[ \sum_{i_1 \neq i_2 \dots \neq i_k} F(x_{i_1}, \dots, x_{i_k}) \right] =: \int \rho_{k, \PNbeta}(y_1, \dots, y_k) F(y_1, \dots, y_k) \dd y_1 \dots \dd y_k
\end{equation*}
for all continuous, compactly supported test functions $F : \R^k \to \R$.
\begin{lemma}
\label{lem:bound_CF}
For each $k \geq 1$, there exists a constant $\Cckb$ such that for all $N \geq 1$:
\begin{equation}
\label{bound:CF}
\sup_{y_1, \dots, y_k} \rho_{k, \PNbeta}(y_1, \dots, y_k) \leq \Cckb.
\end{equation} 
\end{lemma}

\section{Auxiliary infinite-volume results}
\label{sec:prelimINF}
\subsection{Upgrading the convergence}
\label{sec:PrecisionAS}
For a limiting point process $\PI$ to be a weak limit point of $\PpNx$ means (by definition) that for all $f : \Conf \to \R$ which is:
\begin{enumerate}
	\item Bounded,
	\item Continuous on $\Conf$,
\end{enumerate}
we can pass to the limit along the corresponding sub-sequence and write:
\begin{equation}
\label{ConvWeak}
\E_{\PI}[f] = \lim_{N \to \infty} \E_{\PpNx}[f].
\end{equation}
In this section, we upgrade \eqref{ConvWeak} in two ways:
\begin{enumerate}
	\item We replace the boundedness of $f$ by a  much weaker condition called “exp-tame”, see \eqref{eq:ExpTame}.
	\item We extend the convergence to local (but non necessarily continuous) functions on $\Conf$.
\end{enumerate}

\begin{definition}[Tame and $\exp$-tame functions]
A function $f: \Conf \to \R$ is called \emph{tame} when there exists a disk $\DD(x, \ell)$, and $b > 0$, such that:
\begin{equation*}
|f(\bX)| \leq b(1 + \Points(\bX, \DD(x,\ell))) \text{ for all } \bX \in \Conf.
\end{equation*}
By analogy, we will say that a function $f : \Conf \to \R$ is $\exp$-tame if there exists a disk $\DD(x,\ell)$ such that:
\begin{equation}
\label{eq:ExpTame}
|f(\bX)| \leq \exp\left( \frac{1}{\Ccb} \Points^2(\bX, \DD(x,\ell)) \right),
\end{equation}
where $\Ccb$ is the constant appearing in Lemma \ref{lem:points} (so technically speaking the definition  depends on $\beta$).
\end{definition}

\begin{definition}[Local, exp-local topology]
The \emph{local} topology on $\Pr(\Conf)$ is defined as the coarsest topology such that the maps $\Pp \mapsto \E_{\Pp}[f]$ are continuous for all test functions $f$ that are \emph{local} and \emph{tame}. This is already strictly finer than the weak topology. 

By analogy, we define the \emph{exp-local} topology as the coarsest topology such that the maps $\Pp \mapsto \E_{\Pp}[f]$ are continuous for all test functions $f$ that are \emph{local} and \emph{exp-tame}.
\end{definition}
Proving convergence in the exp-local topology means that we can pass much more test functions to the limit. We proceed in two steps.

\subsubsection*{First improvement: from bounded to exp-tame test functions.}
We start by proving the following control on the number of points in the infinite-volume limit.
\begin{lemma}
\label{lem:ControlPointsInfini}
For any $x \in \R^2$ and $\ell \geq \rho_\beta$, we have:
\begin{equation}
\label{ControlPointsInfini}
\log \E_{\PI}\left[\exp\left( \frac{1}{\Ccb} \Points^2(\bX, \DD(x,\ell)) \right)  \right] \leq \Ccb \ell^4.
\end{equation}
\end{lemma}
\begin{proof}[Proof of Lemma \ref{lem:ControlPointsInfini}]
For $x \in \R^2$ and $\ell > 0$, let $F_{x, \ell}$ be a continuous function on $\Conf$ such that:
\begin{equation*}
\Points(\cdot, \DD(x,\ell)) \leq F_{x, \ell} \leq \Points(\cdot, \DD(x,2\ell)).
\end{equation*}
By a simple truncation argument and weak convergence, we have:
\begin{multline*}
\E_{\PI}\left[\exp\left( \frac{1}{\Ccb} \Points_{x, \ell}(\bX) \right)  \right] \leq \E_{\PI}\left[\exp\left( \frac{1}{\Ccb} F^2_{x, \ell}(\bX) \right)  \right] \\
\leq \liminf_{N \to \infty} \E_{\PpNx} \left[\exp\left( \frac{1}{\Ccb} F^2_{x, \ell}(\bX) \right)  \right]
\leq  \liminf_{N \to \infty} \E_{\PpNx}\left[\exp\left( \frac{1}{\Ccb} \Points^2(\bX, \DD(x,2\ell)) \right)  \right].
\end{multline*}
Note that, by definition, the random variable $\Points(\bX, \DD(x,2\ell))$ where $\bX$ follows $\PpNx$ has the same distribution as $\Points(\bXN, \DD(\x + x,2\ell))$. Since $\x$ is “in the bulk” in the sense of \eqref{def:bulk} and $x, \ell$ are fixed, for $N$ large enough $(\x + x, \ell)$ satisfies \eqref{condi:LL} and we may apply Lemma \ref{lem:points}. Taking the log, we obtain \eqref{ControlPointsInfini}.
\end{proof}

Once we know that \eqref{ControlPointsInfini} holds, we deduce the following from an elementary truncation argument.
\begin{lemma}
\label{lem:ExpTame}
If $f : \Conf \to \R^2$ is continuous and $\exp$-tame, then (along the corresponding subsequence):
\begin{equation}
\label{exptamealong}
\E_{\PI}[f] = \lim_{N \to \infty} \E_{\PpNx}[f].
\end{equation}
\end{lemma}

\subsubsection*{Second improvement: from continuous to local test functions.}
\begin{lemma}
\label{lem:upgrading}
If $f : \Conf \to \R^2$ is local and $\exp$-tame, then (along the corresponding subsequence):
\begin{equation*}
\E_{\PI}[f] = \lim_{N \to \infty} \E_{\PpNx}[f].
\end{equation*}
\end{lemma}
This is a bit trickier because local test functions may have no “regularity” whereas test functions used for the weak convergence are continuous with respect to the vague topology on $\Conf$, so we apply a regularization argument and use the boundedness of correlation functions given by \cite{thoma2022overcrowding} to control the error. We postpone the proof of Lemma \ref{lem:upgrading} to Section \ref{sec:ProofUpgrading}. One simple byproduct of the proof, which we record below, is the observation that boundedness of correlation functions can be passed to the limit:
\begin{lemma}
\label{lemma:boundedCorrelInfini}
For all $k \geq 1$ and for some constants depending on $k$ and $\beta$, we have:
\begin{equation}
\label{bound:CF_infinite_record}
\sup_{y_1, \dots, y_k} \rho_{k, \PI}(y_1, \dots, y_k) \leq \Cc_{k, \beta}.
\end{equation} 
\end{lemma}

\subsubsection*{First consequences}
We have already passed the exponential moments $\eqref{eq:bound_points_finite}$ of the number of points to the limit, see \eqref{ControlPointsInfini}. As a consequence of Lemma \ref{lem:upgrading}, we can also guarantee that the estimates \eqref{eq:bound_fluct_Lipschitz} (exponential moments of fluctuations for Lipschitz functions) and \eqref{eq:bound_Fluct_smooth} (exponential moments of fluctuations for smooth functions) remain valid for $\PI$.
\begin{corollary} Let $x$ be a point in $\R^2$ and $\ell \geq \rho_\beta$ be some lengthscale. 

\textbf{1. Lipschitz test functions}
Let $\Lell$ be the set of all functions that are $\ell^{-1}$-Lipschitz and compactly supported on the disk $\DD(x, 4 \ell)$. We have:
\begin{multline}
\label{eq:bound_Fluct_LIPINF}
\log \E_{\PI} \left[ \sup_{\varphi \in \Lell} e^{\frac{1}{\Cc_\beta} \Fluct^2[\varphi]}  \right] \leq \Ccb \ell^2, \quad \log \E_{\PI} \left[ \sup_{\varphi \in \Lell} e^{\Fluct[\varphi]}  \right] \leq \Ccb \ell, \\ \E_{\PI}\left[ \sup_{\varphi \in \Lell}\left|\Fluct[\varphi]\right| \right] \leq \Ccb \ell.
\end{multline}

\textbf{2. Smooth statistics}
Let $\varphi$ be a test function of class $C^3$, compactly supported in $\DD(x, 4 \ell)$. Assume that $|\varphi|_{\kk} \leq \Cc_\varphi \ell^{-\kk}$ for $\kk = 1,2, 3$ for some constant $\Cc_\varphi$. Then for $|s| \leq \frac{\ell^2}{\Cc_\beta \Cc_\varphi}$ we have
\begin{equation}
\label{eq:bound_Fluct_smoothINF}
\left|\log \E_{\PI} \left[ e^{s \Fluct[\varphi]}  \right]\right| \leq \Cc_\beta \left(|s| \Cc_\varphi + s^4 \Cc_\varphi^4 \right), \quad \E_{\PI}\left[ \left| \Fluct[\varphi] \right| \right] \leq \Cc'_\beta \Cc_\varphi.
\end{equation}
\end{corollary}
\begin{proof}
Observe that e.g. $\bX \mapsto \sup_{\varphi \in \Lell} e^{\frac{1}{\Cc_\beta} \Fluct^2[\varphi](\bX)}$ is local and exp-tame.
\end{proof}

\section{Number-rigidity: Proof of Theorem \ref{theo:NR}}
\label{sec:numberRigProof}
Number-rigidity of the limit points $\PI$ can be deduced, a posteriori, as a consequence of the fact that those processes satisfy \eqref{DLR} with an interaction potential (here, the logarithmic one) tending to $-\infty$ at infinity, see the corresponding statement \cite[Thm.~3.18]{dereudre2021dlr} (this is how rigidity is proven in \cite{dereudre2021dlr} for Sine-$\beta$ i.e. for the limit point process of the \emph{one}-dimensional log-gas). We give here a more direct proof following the original strategy of \cite{Ghosh_2017} and using controls on fluctuations of smooth linear statistics for the 2DOCP. This does not use the DLR framework.

\begin{proof}[Proof of Theorem \ref{theo:NR}]
Let us prove that the number of points in the unit disk $\DD_1$ is rigid for $\PI$ i.e. coincides $\PI$-a.s. with a measurable function of the configuration in $\R^2 \setminus \DD_1$ (the proof for any other disk would be identical). The strategy of Ghosh-Peres consists in showing that for all $\epsilon > 0$, there exists a measurable function $\Phi_\epsilon : \R^2 \to \R$ such that:
\begin{enumerate}
	\item $\Phi_\epsilon \equiv 1$ on $\DD_1$.
	\item $\Var_{\PI}[\Fluct[\Phi_\epsilon]] \leq \epsilon$.
\end{enumerate}
The elementary argument of \cite[Thm. 6.1]{Ghosh_2017} shows that this is indeed sufficient to prove number rigidity. In order to construct such $\Phi_\epsilon$'s will rely on the fact that fluctuations of smooth linear statistics for the 2DOCP are well-understood (\cite{leble2018fluctuations,MR4063572,serfaty2020gaussian}).

\paragraph{Step 1: An auxiliary test function.}
\newcommand{\phiep}{\varphi_{\epsilon}}
\newcommand{\phiepl}{\varphi_{\epsilon, \ell}}
For $\epsilon \in (0,1)$, let $\phiep : \R^2 \to [0,1]$ be a smooth, radially symmetric, compactly supported function such that:
\begin{equation}
\label{eq:phiep_support}
\phiep(x) = 1 \text{ for $|x| \leq 1$}, \quad \phiep(x) = 0 \text{ for $|x| \geq 2 e^{1/\epsilon}$},
\end{equation}
and obeying the following controls on its derivatives for all $x \in \R^2$
\begin{equation}
\label{eq:phiep_derivat}
|\phiep|_{\kk, \star}(x) \leq \frac{\Cc \epsilon}{|x|^\kk} \text{ for $\kk = 1, 2, 3$}
\end{equation}
the constant $\Cc$ being uniform with respect to $\epsilon$. The construction of such a function $\phiep$ is elementary, see e.g. \cite[Proof of Theorems 1.1 and 1.3.]{Ghosh_2017}. 

For all $\ell > 1$ we let $\phiepl$ be the rescaled function $\phiepl := \phiep(\cdot / \ell)$. Observe that by construction $\phiepl$ is always equal to $1$ on the unit disk $\DD_1$. It thus suffices to show that:
\begin{equation}
\label{VarPI}
\liminf_{\ell \to + \infty} \Var_{\PI}[\Fluct[\phiepl]] \leq \Cc \epsilon,
\end{equation}
and to then take $\Phi_\epsilon$ as $\phiepl$ for $\ell$ large enough. This boils down to controlling the variance for fluctuations of the smooth linear statistics associated to $\phiepl$.

The test functions $\phiep, \phiepl$ are multi-scale by design and for those functions the existing results (even the most up-to-date ones of \cite{serfaty2020gaussian}) are not stated in an optimal way\footnote{In short: the error term is evaluated in terms of “global” Sobolev norms $W^{1, \infty}$ and $W^{2, \infty}$, which requires a control of derivatives in $L^\infty$, instead of something more local like $W^{1,2}, W^{2, 1}$.}. One would need to inspect the (difficult and lengthy) proof of \cite{serfaty2020gaussian} in order to state a sharp control on their fluctuations. We can however make things work here with minimal effort because the Ghosh-Peres argument is so flexible. 
\newcommand{\xiel}{\xi_{N, \epsilon, \ell}}
We set $\xiel := \phiepl(\cdot - \x(N))$, so that for all $N \geq 1$
\begin{equation*}
\Var_{\PpNx}[\Fluct[\phiepl](\bX)] = \Var[\Fluct[\xiel](\bXN)].
\end{equation*}

\paragraph{Step 2: CLT results and variance bounds.}
The analysis of \cite[Thm. 3  \& Cor. 2.3]{serfaty2020gaussian}, which generalizes the “mesoscopic” case of \cite[Thm.1]{leble2018fluctuations}, imply that the law of $\Fluct[\xiel](\bXN)$ converges \emph{without normalization} to a Gaussian distribution in the following sense:
\begin{equation*}
\lim_{\ell \to \infty} \limsup_{N \to \infty} \left| \log \EN\left[ \exp\left(\tau \Fluct[\xiel]\right) \right] - \frac{\tau^2}{4 \pi \beta} \int_{\R^2} |\nabla \phiep|^2 \right| = 0,
\end{equation*}
uniformly for $\tau$ in compact subsets of $(-\infty, + \infty)$. In particular, we have:
\begin{equation*}
\lim_{\ell \to \infty} \limsup_{N \to \infty} \left| \Var_{\PNbeta} \left[\Fluct[\xiel] \right] - \frac{1}{4 \pi \beta} \int_{\R^2} |\nabla \phiep|^2 \right| = 0.
\end{equation*}
Observe that $\int_{\R^2} |\nabla \phiep|^2 = \O(\epsilon)$. We deduce that for $\ell$ large enough, as $N \to \infty$ (depending on $\epsilon$): 
\begin{equation*}
\Var_{\PNbeta} \left[\Fluct[\phiepl] \right] \leq \Cc_\beta \epsilon,
\end{equation*} 
which can be passed to the exp-tame limit (as $N \to \infty$) and yields \eqref{VarPI}.

This concludes the proof of Theorem \ref{theo:NR}.
\end{proof}

\section{Making sense of DLR equations: proof of Theorem \ref{theo:zero}}
\label{sec:ProofMaking}
\newcommand{\tMLa}{\widetilde{\mathsf{M}}_\La}
\newcommand{\MLaN}{\mathsf{M}_{\La, N}}
\newcommand{\MLaNp}{\mathsf{M}^{p}_{\La, N}}
\newcommand{\tMLaN}{\widetilde{\mathsf{M}}_{\La, N}}
\newcommand{\tMLaNp}{\widetilde{\mathsf{M}}^{p}_{\La, N}}
\newcommand{\MquL}{\mathsf{M}_\La^{?}}
\newcommand{\bXabs}{\bX^{\mathrm{abs}}}
\subsection{Heuristics.}
Let $\La \subset \R^2$ a compact subset. Giving sense to DLR equations implies defining the interaction between a configuration of points in $\La$ and one in $\bLa$. A naive way would be to set:
\begin{equation}
\label{HeurMove}
\MquL(\bX', \bX) := \iint_{\Lambda \times \bLa}  - \log|x - y|  \dd(\bX' -  \Leb)(x) \dd(\bX - \Leb)(y). 
\end{equation}
Let us explain why this might not make sense. First, we can write this hypothetical quantity as:
\begin{equation}
\label{HeurMove2}
\MquL(\bX', \bX) = \sum_{x \in \bX'} \Fluct\left[- \log|x-\cdot] \1_{\bLa} \right](\bX) - \int_{x \in \La} \Fluct\left[- \log|x-\cdot] \1_{\bLa} \right](\bX ) \dd \Leb(x).
\end{equation}
Now, fix a point $x$ in $\La$ and consider $\Fluct\left[- \log|x-\cdot| \ind_{\bLa} \right](\bX)$, which represents the interaction of $x$ with the configuration outside $\La$ and the infinitely extended background given by the Lebesgue measure on $\bLa$. Disregarding the indicator function, $\log|x-\cdot|$ is fairly smooth, and we know (by \cite{leble2018fluctuations,MR4063572,serfaty2020gaussian}) that the typical size of the linear statistics for a smooth test function $\varphi$ is given (in terms of standard deviation) by $\left(\int_{\R^2} |\nabla \varphi|^2\right)^\hal$. Here, this would not quite give a finite result because $\nabla \log|x-\cdot|$ (barely) fails to be in $L^2$ near infinity. 

In order to make a gain (e.g. some additional decay), we would like to use the cancellation between $\bX'$ and the Lebesgue measure on $\La$ in \eqref{HeurMove}, \eqref{HeurMove2}. However $\Points(\bX', \La)$ and $\Leb(\La)$ might not be equal as there is no reason for having a perfect charge balance within $\La$. 

To overcome this difficulty, we consider a slightly different quantity and introduce some “abstract” \emph{reference point configuration $\bXabs$ in $\La$ that has the same number of points as $\bX'$}. The actual choice does not matter, so we simply take
\begin{equation*}
\bXabs := \Points(\bX', \La) \times \delta_0,
\end{equation*}
i.e. we place the correct number of points, all of them located at the origin, and we then define (cf. \eqref{HeurMove2}):
\begin{equation}
\label{HeurMove3}
\tMLa(\bX', \bX) = \sum_{x \in \bX'} \Fluct\left[- \log|x-\cdot] \1_{\bLa} \right](\bX) - \sum_{x \in \bXabs} \Fluct\left[- \log|x-\cdot| \1_{\bLa} \right](\bX).
\end{equation}
This in turn can be rewritten as:
\begin{multline}
\label{def:tMLa}
\tMLa(\bX', \bX) = \sum_{x \in \bX'} \Fluct \left[- \log|x-\cdot| \1_{\bLa} \right](\bX) - \Points(\bX', \La) \Fluct\left[- \log|\cdot| \1_{\bLa} \right](\bX) \\
= \iint_{\Lambda \times \bLa}  \left(- \log|x - y| + \log|y| \right)  \dd \bX'(x) \dd\left(\bX - \Leb\right)(y),
\end{multline}
and we will mostly work with that last expression. The connection with our naive guess \eqref{HeurMove} is that:
\begin{equation*}
\tMLa(\bX', \bX) = \MquL(\bX', \bX) + \left(\int_{x \in \La} \Fluct\left[- \log|x-\cdot] \1_{\bLa} \right](\bX ) \dd \Leb(x) - \sum_{x \in \bXabs} \Fluct\left[- \log|x-\cdot| \1_{\bLa} \right](\bX)\right),
\end{equation*}
the parenthesis in the right-hand side being a constant with respect to the position of the points in $\bX'$ (but it depends on the number of points of $\bX'$ in $\bLa$, which is why we eventually obtain \emph{canonical} DLR equations). Recall that in statistical physics, shifting energies by a constant does not affect the probability distribution, as this constant gets absorbed in the partition function. One may think of the manipulations presented above as “shifting energies by a global, possibly infinite, constant”. 

Note that $\tMLa$ still needs to be defined in a proper way because $\bLa$ is unbounded (see \eqref{def:tMLa}). Once we know that $\tMLa$ exists, and provided $\bX'$ and $\bX$ have the same number of points in $\La$, we define (cf. \eqref{eq:MLa}):
\begin{equation}
\label{reDefMLa}
\MLa(\bX', \bX) = \tMLa(\bX', \bX) - \tMLa(\bX, \bX) = \iint_{\La \times \bLa} -\log|x-y| \dd  \left(\bX' - \bX\right)(x) \dd \left(\bX - \Leb \right)(y).
\end{equation}
which is the cost of “moving points from $\bX$ to $\bX'$ within $\La$”. The quantity we will eventually use is $\MLa$, but it is convenient to first prove results on $\tMLa$ because of its simple expression \eqref{def:tMLa}. 

\subsection{Partial move functions.}
Recall that $\chi = (\chi_i)_{i \geq 1}$ is the dyadic partition of unity introduced in Section \ref{sec:notation}. For $p \geq 1$, define the “partial move function” $\tMLa^p$ as (cf. \eqref{def:tMLa})
\begin{equation}
\label{def:tMLap}
\tMLa^p(\bX', \bX) := \iint_{\Lambda \times \bLa}  \left(- \log|x - y| + \log|y| \right)  \dd \bX'(x) \sum_{i=0}^p \chi_i(y)  \dd\left(\bX - \Leb\right)(y),
\end{equation}
and define similarly $\MLa^p$ as (cf. \eqref{reDefMLa}):
\begin{equation}
\label{def:MLap}
\MLa^p(\bX', \bX) := \iint_{\La \times \bLa} -\log|x-y| \dd  \left(\bX' - \bX \right)(x) \sum_{i=0}^p \chi_i(y) \dd \left(\bX - \Leb \right)(y).
\end{equation}
We will also use the “finite-$N$” versions, with the exterior restricted to $\LN$, namely:
\begin{equation}
\tMLaNp(\bX', \bX) := \iint_{\Lambda \times \bLa}  \left(- \log|x - y| + \log|y| \right) \1_{\LN}(y)  \dd \bX'(x) \sum_{i=0}^p \chi_i(y)  \dd\left(\bX - \Leb\right)(y),
\end{equation}
ad well as:
\begin{equation}
\label{def:MLaNp}
\MLaNp(\bX', \bX) := \iint_{\La \times \bLa} -\log|x-y| \1_{\LN}(y) \dd  \left(\bX' - \bX \right)(x) \sum_{i=0}^p \chi_i(y) \dd \left(\bX - \Leb \right)(y).
\end{equation}

\begin{remark}
\label{rem:Nplarge}
Note that for $p$ fixed, if $N \geq 2^{p+2}$, then $\MLaNp$ and $\MLa^p$ coincide, as well as $\tMLaNp$ and $\tMLa^p$.
\end{remark}

\newcommand{\phiix}{\varphi_{i,x}}
\newcommand{\Rem}{\mathrm{Rem}}

\subsection{Almost sure existence of total move functions, and control on remainders.}
\begin{proposition}
\label{prop:Zero}
For all bounded Borel set $\La \subset \R^2$, the following holds.

\textbf{A. Infinite volume.}
For $\PI$-almost every $\bX$:
\begin{enumerate}
    \item For all point configuration $\bX'$ in $\La$, the limit $\lim_{p \to \infty} \tMLa^p(\bX', \bX)$ of \eqref{def:tMLap} exists. We denote it by $\tMLa$.
    \item For all $n \geq 0$, we control the truncation error among point configurations with $n$ points in $\La$:
    \begin{equation}
    \label{uniformtMLa}
    \lim_{p \to \infty} \E_{\PI}\left[\sup_{\bX' \in \Conf(\La, n)} |\tMLa(\bX', \bX) - \tMLa^p(\bX', \bX)| \right] = 0
    \end{equation}
\end{enumerate}

\textbf{B. Finite volume.}
For all $n \geq 0$, we control the truncation error in finite volume among point configurations with $n$ points in $\La$:
\begin{equation}
\label{uniformtMLa_N}
\lim_{p \to \infty} \lim_{N \to \infty}  \E_{\PpNx} \left[\sup_{\bX' \in \Conf(\La, n)} |\tMLaN(\bX', \bXN) - \tMLaNp(\bX', \bXN)| \right] = 0.
\end{equation}
\end{proposition}
Item A.1. ensures the almost sure existence of $\tMLa$, which will eventually allow us to \emph{make sense of DLR equations} (see Section \ref{sec:proof0}). Items A.2 and B. provide uniform controls on the remainder $\tMLa - \tMLa^p$ and the corresponding finite-$N$ quantity, which will be crucial when \emph{proving that DLR equations hold}, by allowing to localize the move functions. Indeed, $\tMLa^p(\bX', \cdot)$ is a local function on $\Conf$, whereas $\tMLa(\bX', \cdot)$ is not.

\begin{proof}[Proof of Proposition \ref{prop:Zero}]
Recall (see Section \ref{sec:notation}) that $\chi = (\chi_i)_{i\geq 0}$ is a dyadic partition of unity such that each function $\chi_i$ is of class $C^3$, compactly supported in the annulus $\DD(2^{i+1}) \setminus \DD(2^{i-1})$ for $i \geq 1$, and satisfies $|\chi_i|_{\kk} \leq \Cc_\chi 2^{-i \kk}$ for $\kk = 1,2,3$ with some constant $\Cc_\chi$ uniform in $i$.  

\paragraph{Re-writing each layer.}
For $x \in \La$ and $i \geq 0$, let $\phiix$ be the following function:
\begin{equation}
\label{defphiix}
\phiix : y \mapsto \left( -\log |x-y| + \log |y| \right) \chi_i(y).
\end{equation}
For all $p \geq 1$ and all $\bX, \bX'$ we can write the partial move function $\tMLa^p$ (see \eqref{def:tMLap}) as:
\begin{equation}
\label{rewritetMLap}
\tMLa^p(\bX', \bX) := \iint_{\La \times \R^2} \sum_{i=0}^p \phiix(y) \1_{\bLa}(y) \dd \bX'(x)  \dd \left(\bX - \Leb \right)(y)
= \sum_{i=0}^p \sum_{x \in \bX'} \Fluct[\phiix \1_{\bLa}](\bX),
\end{equation} 
as well as its finite-$N$ counterpart (note the extra constraint due to $\1_{\LN}$):
\begin{equation}
\label{rewritetMLaNp}
\tMLaNp(\bX', \bX) = \sum_{i=0}^p  \sum_{x \in \bX'} \Fluct[\phiix \1_{\bLa} \1_{\LN}](\bX).
\end{equation}

Let $I \geq 1$ be an index such that $\La \subset \DD(2^{I-4})$ (it exists because $\La$ is bounded), the point being that if $i \geq I$, then the support of $\chi_i$ is far (at least one dyadic scale away) from $\La$, and in particular we have $\phiix \ind_{\bLa} = \phiix$, so we may forget the indicator function.

For $x \in \La$, $y$ fixed and $i \geq I$, we apply a Taylor's expansion to $-\log|y-\cdot|$ around $0$ and we decompose $\phiix(y)$ (see \eqref{defphiix}) as: 
\begin{equation}
\label{phiiexpan}
\phiix(y) =  \left\langle x, \frac{y}{|y|^2} \chi_i(y) \right\rangle  + \Rem_i(x,y).
\end{equation}
 The first-order term in \eqref{phiiexpan} can be written as $\langle x, J_i(y) \rangle$ with $J_i := y \mapsto \frac{y}{|y|^2} \chi_i(y)$, which is a smooth function that satisfies:
\begin{equation}
\label{Ji}
|J_i|_\kk \leq \Cc_\chi 2^{-i (\kk+1)} \ \kk = 0, 1, 2, 3. 
\end{equation}
On other hand, the “second-order” remainder term $\Rem_i(x,\cdot)$ satisfies:
\begin{equation}
\label{Remi}
|\Rem_i(x,\cdot)|_{\kk} \leq \Cc_\La 2^{-i (\kk + 2)} \text{ and in particular } |\Rem_i(x,\cdot)|_{\1} \leq \Cc_\La 2^{-3i}.
\end{equation} 
For $i \geq I$, we are left with:
\begin{equation}
\label{RewriteFluctphiix}
\Fluct[\phiix \1_{\bLa}] = \left\langle x, \Fluct[J_i] \right\rangle + \Fluct[\Rem_i(x,\cdot)],
\end{equation}
where:
\begin{itemize}
    \item $J_i$ is a radially symmetric, smooth, compactly supported function (on a disk of radius $2^{i+2}$) that does not depend on $x$ and satisfies \eqref{Ji}.
    \item $\Rem_i(x,\cdot)$ is a Lipschitz, compactly supported function (on a disk of radius $2^{i+2}$), that does depend on $x$ but satisfies \eqref{Remi} uniformly with respect to $x$ in $\La$.
\end{itemize}

\paragraph{Infinite-volume case: existence of move functions.}
Let us first prove the almost sure existence of total move functions. We have, in view of \eqref{rewritetMLap} and introducing the index $I$ chosen above:
\begin{equation}
\label{I2nd}
\tMLa^p(\bX', \bX) = \sum_{i<I} \sum_{x \in \bX'} \Fluct[\phiix \1_{\bLa}](\bX) + \sum_{i=I}^p \sum_{x \in \bX'} \Fluct[\phiix \1_{\bLa}](\bX).
\end{equation} 
To prove item A.1. it is enough to prove the existence of a limit as $p \to \infty$ of the second sum, \emph{for almost every $\bX$ and for all $\bX'$}. Indeed, the boundedness of the $1$-point correlation function implies that for $\PI$-a.e. point configuration $\bX$ in $\bLa$ one has $\inf_{x \in \bX} \dist(x, \La) > 0$, and thus, since $\varphi_{i,x}$ is bounded on compact subsets of $\bLa$):
\begin{equation}
\label{MoinsI}
\left|\sum_{i < I} \sum_{x \in \bX'} \Fluct[\phiix \1_{\bLa}](\bX) \right| \leq \Points(\bX', \La) \times \Cc(\bX, \La, I) \text{ is finite for all $\bX'$ in $\bLa$.}
\end{equation}
 We now focus on the second sum in \eqref{I2nd}. Using \eqref{RewriteFluctphiix} we write:
\begin{equation}
\label{sumiP}
\sum_{i=I}^p \sum_{x \in \bX'} \Fluct[\phiix \1_{\bLa}](\bX) =
\sum_{i=I}^p \sum_{x \in \bX'} \left\langle x, \Fluct[J_i](\bX) \right\rangle +  \sum_{i=I}^p \sum_{x \in \bX'} \Fluct[\Rem_i(x,\cdot)(\bX)]
\end{equation}
The first term in the right-hand side of \eqref{sumiP} can be trivially bounded by
\begin{equation}
\label{JiA}
\sum_{i=I}^p \sum_{x \in \bX'} \left\langle x, \Fluct[J_i](\bX) \right\rangle \leq \Points(\bX', \La) \times \Cc_\La \times \sum_{i=I}^p \left| \Fluct\left[J_i\right](\bX)\right|,
\end{equation}
where the last term does not depend on $\bX'$. Applying the control \eqref{eq:bound_Fluct_smoothINF} on fluctuations of smooth statistics to $J_i$ that satisfies \eqref{Ji}, we get for $i \geq I$:
\begin{equation}
\label{cfFluctJi}
\E_{\PI} \left[ \left|\Fluct[J_i](\bX)\right|\right] \leq \Cc_{\beta, \chi} 2^{-i},
\end{equation} 
and we thus obtain that:
\begin{equation*}
\E_{\PI} \left[ \sum_{i=I}^{+\infty} \left| \Fluct\left[J_i\right](\bX)\right| \right] < + \infty,
\end{equation*}
which ensures in view of \eqref{JiA} that, for $\PI$-a.e. $\bX$ and for all $\bX'$, the quantity
\begin{equation}
\label{MainOrder}
\lim_{p \to \infty} \sum_{i=I}^p \sum_{x \in \bX'} \left\langle x, \Fluct[J_i](\bX) \right\rangle  \text{ exists and is finite.}
\end{equation}

On the other hand, in view of our control \eqref{Remi} on $\Rem_i$, the second term in the right-hand side of \eqref{sumiP} can be bounded by
\begin{equation}
\label{RemiA}
\left|\sum_{i=I}^p \sum_{x \in \bX'} \Fluct[\Rem_i(x,\cdot)(\bX)] \right| \leq \Points(\bX', \La) \times \Cc_{\chi, \La} \times \sum_{i=I}^p 2^{-2i} \sup_{f_i \in \mathcal{L}_i} |\Fluct[f_i](\bX)|,
\end{equation}
where $\mathcal{L}_i$ is the set of all Lipschitz functions that are $2^{-i}$-Lipschitz and compactly supported in $\DD(2^{i+2})$. Using the uniform control \eqref{eq:bound_Fluct_LIPINF} on fluctuations of Lipschitz statistics, we get for $i \geq 1$:
\begin{equation}
\label{cfFluctfi}
\E_{\PI} \left[ \sup_{f_i \in \mathcal{L}_i} |\Fluct[f_i](\bX)| \right] < \Cc_{\beta, \chi} 2^{i}, \text{ and thus } \E_{\PI} \left[ \sum_{i=I}^p 2^{-2i} \sup_{f_i \in \mathcal{L}_i} |\Fluct[f_i](\bX)|  \right] \leq \Cc_{\beta, \chi} 2^{-I} < + \infty. 
\end{equation}
This ensures that for $\PI$-a.e. $\bX$ and for all $\bX'$, the quantity
\begin{equation}
\label{SecondOrder}
\lim_{p \to \infty} \sum_{i=I}^p \sum_{x \in \bX'} \Fluct[\Rem_i(x,\cdot)(\bX)]  \text{ exists and is finite.}
\end{equation}
Combining \eqref{MoinsI}, \eqref{MainOrder} and \eqref{SecondOrder} proves item A.1.

\begin{remark}
\label{rem:NeededForMove}
The proof of item A.1 only required three ingredients:
\begin{enumerate}
    \item Boundedness of the $1$-point correlation function as in \eqref{bound:CF_infinite_record} (used for \eqref{MoinsI}).
    \item Uniform control on Lipschitz statistics as in \eqref{eq:bound_Fluct_LIPINF} (used for \eqref{SecondOrder}).
    \item Control on smooth statistics as in \eqref{eq:bound_Fluct_smoothINF} (used for \eqref{MainOrder}).
\end{enumerate}
It thus readily applies to all point processes satisfying those three conditions (see Section \ref{sec:good} below).
\end{remark}

\paragraph{Infinite-volume case: control on remainders.} The previous computation (combining \eqref{sumiP}, \eqref{JiA}, \eqref{RemiA}) has also shown that, for $p \geq I$:
\begin{equation*}
\left|\tMLa(\bX', \bX) - \tMLa^p(\bX', \bX) \right| \leq \Points(\bX', \La) \times \Cc_{\La, \chi} \times \left( \sum_{i \geq p} \left|\Fluct\left[J_i\right](\bX)\right|  + 2^{-2i} \sup_{f_i \in \mathcal{L}_i} |\Fluct[f_i](\bX)| \right),
\end{equation*}
and that (the number $\Points(\bX', \La)$ being fixed but arbitrary), the right-hand side goes to $0$ in expectation as $p \to \infty$, uniformly with respect to $\bX'$. This proves \eqref{uniformtMLa}.

\paragraph{Finite-volume case: control on remainders.} Existence of move functions poses no problem in finite volume. However, when controlling the truncation error $\tMLaN - \tMLaNp$, there is a subtle difference due to the presence of the additional indicator $\1_{\LN}$ within the test function (see \eqref{rewritetMLaNp}). This only affects the “last” layer, for which $\supp \chi_i$ contains $\partial \LN$, which corresponds to 
\begin{equation*}
i = \log \dist(\x(N), \partial \LN) + \O(1) \geq \hal \log N + \O(1).
\end{equation*}
For this index, we rely on the “hard edge versions stated in Lemma \ref{lem:LipschitzHardWall} and Lemma \ref{lem:SmoothHardWall}, which (together with the controls \eqref{Remi} on $\Rem_i$ and \eqref{Ji} on $J_i$) ensure that:
\begin{equation*}
\EN\left[ \sup_{x \in \La} \left| \Fluct\left[\Rem_i(x,\cdot) \1_{\LN}\right](\bX)\right| \right] \leq \Cc_{\beta, \La} N^{-\hal}, \quad \EN \left[ \left|\Fluct\left[J_i \1_{\LN} \right](\bX)\right| \right] = o(1),
\end{equation*}
with error terms depending on $\beta, \La$ and the choice of the sequence $\x = \x(N)$. This last estimate is not as good as one could hope for, but it is enough to obtain \eqref{uniformtMLa_N}.
\end{proof}

\subsection{Consequence for the “true” move function \texorpdfstring{$\MLa$}{}: proof of Theorem \ref{theo:zero}}
\label{sec:proof0}
Recall that for $\bX$ in $\Conf(\R^2)$ and $\bX'$ in $\Conf(\La, \Points(\bX, \La))$ we define $\MLa(\bX', \bX)$ as 
(see \eqref{def:MLap}):
\begin{equation*}
\MLa(\bX', \bX) = \tMLa(\bX', \bX) - \tMLa(\bX, \bX) = \lim_{p \to \infty} \MLa^p(\bX', \bX).
\end{equation*}
As an easy consequence of Proposition \ref{prop:Zero}, we obtain for all bounded Borel set $\La \subset \R^2$:
\begin{corollary}
For $\PI$-almost every $\bX$ and for all $\bX'$ in $\Conf(\La, \Points(\bX, \La))$, the limit 
\begin{equation*}
\MLa(\bX', \bX) := \lim_{p \to \infty} \MLa^p(\bX', \bX)
\end{equation*} 
exists and is finite, which proves Theorem \ref{theo:zero}.
\end{corollary}

\subsection{Discussion on the “last layer”.} 
\label{sec:LastLayer}
In the last step of the proof of Proposition \ref{prop:Zero}, there would be a slight difference of treatment in one were to choose the RMT convention of a soft confinement of the particles instead of a hard wall along $\partial \LN$. For smooth functions intersecting the boundary, one would rather apply the “Macroscopic boundary case” of \cite[Thm. 1]{leble2018fluctuations} (which is a bit simpler to obtain in the radial case than in the general setting of \cite{leble2018fluctuations}), and use it to bound $\EN \left[ \left|\Fluct\left[J_i \right](\bX)\right| \right]$.

In the setup that we have chosen here, it seems quite difficult to avoid relying on fine results, e.g. the hyperuniformity statement of \cite{leble2021two}. Indeed, the presence of $\1_{\LN}$ prevents us from applying “smooth” techniques. 

Another option would of course be to place the 2DOCP on a sphere (as is sometimes done in the physics literature) instead of a disk, and to get rid of boundary effects altogether.

\section{The canonical DLR equations: proof of Theorem \ref{theo:DLR}}
\label{sec:DLR_expression}
Once the existence of “move” functions is established, together with truncations estimates, the proof of \eqref{DLR} follows a classical route, see e.g. closely related examples in \cite[Sec. 2]{dereudre2021dlr} or \cite[Sec. 2.3]{dereudre2023number}.

We first introduce some notation: if $f : \Conf \to \R$ is measurable and bounded, and $\La \subset \R^2$ is a bounded Borel set, we let
\begin{equation*}
	f_\Lambda(\bX) 
	= 
\frac{1}{\KLbeta(\bX)} \int_{\Conf(\La, \Points(\bX, \La))} f\left(\bX' \cup \bX_{\bLa} \right) \exp\left(-\beta(\FL(\bX') + \MLa(\bX', \bX)\right) \dd \Bin_{\La, \Points(\bX, \La)}(\bX'),
\end{equation*}
where $\KLbeta(\bX)$ is the corresponding partition function. For $p \geq 1$, we define similarly:
\begin{equation*}
    \fLp(\bX) 
    =  \frac{1}{\KLbetap(\bX)} \int_{\Conf(\La, \Points(\bX, \La))} f\left(\bX' \cup \bX_{\bLa}  \right) \exp\left(-\beta(\FL(\bX') + \MLa^p(\bX', \bX)\right) \dd \Bin_{\La, \Points(\bX, \La)}(\bX'),
\end{equation*}
where $\KLbetap(\bX)$ is the “truncated” partition function, namely:
\begin{equation}
\label{def:KLbetap}
\KLbetap(\bX) :=  \int_{\Conf(\La, \Points(\bX, \La))} \exp\left(-\beta(\FL(\bX') + \MLa^p(\bX', \bX)\right) \dd \Bin_{\La, \Points(\bX, \La)}(\bX').
\end{equation}
Proving the DLR equations thus consists in showing that
\begin{equation}
\label{eq:DLREPI}
\text{For all bounded Borel $\La \subset \R^2$, for all bounded measurable $f : \Conf \to \R$}, \ \E_{\PI}[f] = \E_{\PI}[f_\La].
\end{equation}

We also introduce: 
\begin{equation*}
    \fNL(\bX) 
    = 
\frac{1}{\KLbeta(\bX)} \int_{\Conf(\La, \Points(\bX, \La))} f\left(\bX' \cup \bX_{\bLa} \right) \exp\left(-\beta(\FL(\bX') + \MLaN(\bX', \bX)\right) \dd \Bin_{\La, \Points(\bX, \La)}(\bX').
\end{equation*}
Recall (see Remark \ref{rem:Nplarge}) that for $p$ fixed and $N$ large enough, $\MLaN^p$ and $\MLa^p$ are the same quantity, so it will not be necessary to introduce $f_{\La, N}^p$.

\subsection{Truncation errors}
For $p \geq 1$ and $\delta >0$, define the event $\mathsf{A}_\Lambda^p(\delta) \subset \Conf$ as:
\begin{equation}
	\mathsf{A}_\Lambda^p(\delta) = \left\{
	\bX \in \Conf : \sup_{\bX' \in \Conf(\Lambda, \Points(\bX, \La))} \left| \MLa(\bX', \bX) - \MLa^p(\bX', \bX) \right| \leq \delta.
	\right\}
\end{equation}
If $\bX$ is in $\mathsf{A}_\Lambda^p(\delta)$, we can approximate the total move functions $\MLa(\bX', \bX)$ by their local versions $\MLa^p(\bX', \bX)$ at (dyadic) scale $p$ while making an error of size at most $\delta$ uniformly with respect to $\bX'$. 

Define $ \mathsf{A}_{\Lambda,N}^p(\delta)$ similarly with $\MLaN$ instead of $\MLa$.

\begin{lemma}
\label{lem:good_move}
	For all $\delta > 0$ and $\varepsilon>0$, there exists $p$ such that:
\begin{enumerate}
    \item 
    \begin{equation*}
        \PI\left(\mathsf{A}_\Lambda^p(\delta)\right) \geq 1 - \varepsilon. 
    \end{equation*}
    \item  For all $N$ large enough:
    \begin{equation*}
        \PpNx\left(\mathsf{A}_{\Lambda, N}^p(\delta)\right) \geq 1 - \varepsilon. 
    \end{equation*}
\end{enumerate}
\end{lemma}
\begin{proof}
This follows from \eqref{uniformtMLa} or \eqref{uniformtMLa_N} on the one hand and Lemma \ref{lem:ControlPointsInfini} or Lemma \ref{lem:points} on the other hand, using Markov's inequality twice: to bound the number of points of $\bX$ in $\La$ up to an event of probability smaller than $\hal \varepsilon$, and to bound the truncation error using Proposition \ref{prop:Zero} for $p$ large enough, again up to an event of probability smaller than $\hal \varepsilon$.
\end{proof}

\begin{lemma}[Using truncations]
	For all $\varepsilon > 0$, if $p$ is large enough:
    \begin{itemize}
        \item  For any bounded measurable function $f$ we have
    \begin{equation}
    \label{IVTrun}       
        \left| \E_{\PI} \left[\fL - \fLp\right]\right| \leq \varepsilon \|f\|_\infty
    \end{equation}
    \item If $N$ is large enough, for any bounded measurable function $f$ we have:
    \begin{equation}
    \label{FVTrun}       
        \left| \E_{\PpNx} \left[\fNL - \fLp\right]\right| \leq \varepsilon  \|f\|_\infty
    \end{equation}
    \end{itemize}
\end{lemma}
\begin{proof} The starting point is to write the difference as 
	\begin{multline}
		\label{eq:diff_truncation}
f_\Lambda(\bX) - \fLp(\bX) = \frac{1}{\KLbeta(\bX)} \int f(\bX' \cup \bX_\Lambda) 
		\left(        e^{-\beta\left(\FLa(\bX') + \MLa(\bX', \bX) \right)}
        -       e^{-\beta\left(\FLa(\bX') + \MLa^p(\bX', \bX) \right)}  \right)
		\dd \Bin_{\Lambda, \Points(\bX, \Lambda)}(\bX')
		\\
		+ \left(
		\frac{1}{\KLbetap(\bX)} - \frac{1}{\KLbeta(\bX)} 
		\right)
		\int f(\bX' \cup \bX_\Lambda) e^{-\beta\left(\FLa(\bX') + \MLa(\bX', \bX)\right)} 
		\dd \Bin_{\Lambda, \Points(\bX, \Lambda)}(\bX'). 
	\end{multline}
	
	If $\bX \in \mathsf{A}_\Lambda^p(\delta)$ then for every $\bX' \in \Conf(\Lambda)$ such that $\Points(\bX',\La) = \Points(\bX, \La)$ we have:
	\begin{equation}
		\label{eq:diff_exp}
		\left|
		e^{-\beta(\FL(\bX') + \MLa(\bX', \bX)} - e^{-\beta(\FL(\bX') + \MLa^p(\bX', \bX)) }
		\right|
		\leq 
		(e^{\beta \delta} -1) e^{-\beta(\FL(\bX') + \MLa^p(\bX', \bX))}, 
	\end{equation}
	and in particular we can control truncation error in the partition functions:
	\begin{equation}
		\label{eq:relativ_diff}
		\left|
		\KLbetap(\bX) - \KLbeta(\bX) 
		\right|
		\leq 
		(e^{\beta \delta} -1) \KLbetap(\bX). 
	\end{equation}	
	Inserting \eqref{eq:diff_exp} and \eqref{eq:relativ_diff} in \eqref{eq:diff_truncation} shows that for every $\bX \in \mathsf{A}_\Lambda^p(\delta)$ we have 
	\begin{equation*}
		\left|
		\fLp(\bX) - \fL(\bX) 
		\right|
		\leq 
		2(e^{\beta \delta} -1) \normf. 
	\end{equation*}
	Finally, we may write:
	\begin{equation*}
		\big| \E_{\PI} (\fLp - \fL) \big|
		\leq 
		2 (e^{\beta \delta} - 1)\normf + 2(1 - \PI(\mathsf{A}_\Lambda^p(\delta))) \normf
	\end{equation*}
	Fixing $\delta$ small enough to have $2(e^{\beta \delta} -1) \leq \frac{\varepsilon}{2}$, and taking $p$ large enough such that $\PI(\mathsf{A}_\Lambda^p(\delta)) \geq 1 - \frac{\varepsilon}{4}$ (by Lemma~\ref{lem:good_move}), we obtain \eqref{IVTrun}. 

    The proof of \eqref{FVTrun}  follows the same lines.
\end{proof}

The DLR equations in infinite volume will be inherited from those satisfied in finite volume. 

\subsection{Finite volume canonical DLR equations}
\begin{proposition}
\label{prop:finite-DLR}
	For all bounded Borel set $\Lambda$ and for all bounded measurable function $f : \Conf \to \R$, provided $N$ is large enough that $\Lambda \subset \LN$, we have: 
	\begin{equation}
    \label{finiteDLR}
		\E[f(\bXN)] = \E[f_{\La,N}(\bXN)]
	\end{equation}
\end{proposition}
\begin{proof}[Proof of Proposition \ref{prop:finite-DLR}]
We make the following simple observations: 
\begin{itemize}
    \item If $\bXN$ is a point configuration in $\LN$ and $\La \subset \LN$, which we decompose as $\bXN = \bX_\La \cup \bX_{\bLa}$, we have:
\begin{equation*}
\FN(\bXN) = \FL(\bX_\La) + \FbL(\bX_{\LN \setminus \La}) + \iint_{\La \times (\LN \setminus \La)} - \log|x-y| \dd \left(\bX - \Leb\right)(x) \dd \left(\bX - \Leb\right)(y),
\end{equation*}
with the notation of \eqref{def:FLa} for local interaction energy.
    \item If $\bX'$ is another point configuration in $\La$, such that $\Points(\bX', \La) = \Points(\bX, \La)$, we have:
  \begin{multline*}
 \MLaN(\bX', \bX) =  \iint_{\La \times (\LN \setminus \La)} - \log|x-y| \dd \left(\bX' - \Leb\right)(x) \dd \left(\bX - \Leb\right)(y) \\ - \iint_{\La \times (\LN \setminus \La)} - \log|x-y| \dd \left(\bX - \Leb\right)(x) \dd \left(\bX - \Leb\right)(y).
  \end{multline*}  
\end{itemize}
The finite-volume canonical DLR equations then follow from simple manipulations of binomial point processes, see e.g. \cite[Prop.~2.8]{dereudre2021dlr} for the proof in the context of $1d$ log-gases or \cite[Prop.~2.10.]{dereudre2023number} for Riesz gases (the precise nature of the pairwise interaction matters very little here).
\end{proof}
The finite-volume DLR equations \eqref{finiteDLR} are stated for $\bXN = \bX_{N, \mathsf{0}}$, namely the local point configuration \emph{seen from the origin}. We can extend them readily for all the local finite point processes $\PpNx$ by shifting the test function $f$ accordingly:
\begin{corollary}
For all bounded Borel set $\La$ and for all bounded measurable function $f : \Conf \to \R$, provided $N$ is large enough\footnote{Observe that if $\La$ is a fixed Borel set and $\x = \x(N)$ is “in the bulk” in the sense of \eqref{def:bulk}, then $\La + x \subset \LN$ for all $N$ sufficiently large.} that $\La + \x \subset \LN$, we have:
    \begin{equation}
    \label{finiteDLR2}
        \E_{\PpNx}[f(\bX)] = \E_{\PpNx}[f_{\La,N}(\bX)].
    \end{equation}
\end{corollary}
\renewcommand{\u}{\vec{u}} \newcommand{\Ppp}{\Pp^+} \newcommand{\Ppm}{\Pp^-}

\subsection{Passing the finite-volume DLR equations to the limit: proof of Theorem \ref{theo:DLR}}
The proof of Theorem \ref{theo:DLR} is then standard. Take $f$ bounded, measurable \emph{and assume moreover that $f$ is local}. We write:
\begin{multline*}
\E_{\PI}[f] - \E_{\PI}[f_\La] = 
\left(\E_{\PI}[f] - \E_{\PpNx}[f]\right) + \left(\E_{\PpNx}[f] - \E_{\PpNx}[f_\La]\right) + \left(\E_{\PpNx}[f_\La] - \E_{\PpNx}[f^p_\La]\right) \\
+ \left(\E_{\PpNx}[f^p_\La] - \E_{\PI}[f^p_\La]\right) + \left(\E_{\PI}[f^p_\La] - \E_{\PI}[f_\La]\right).
\end{multline*}
The first and the second-to-last parentheses in the right-hand side are $o_N(1)$ by definition of local convergence (note that although $f_\La$ is \emph{not} local in general, $f^p_\La$ does remain a local function - which is the whole point of introducing truncations). The second one vanishes for $N$ large enough by the finite-volume DLR equations in the form \eqref{finiteDLR2}.  Moreover, by \eqref{IVTrun}, \eqref{FVTrun} we know that if $p$ and $N$ are large enough the third and fourth parentheses can be made arbitrarily small. We deduce that:
\begin{equation*}
\E_{\PI}[f]  = \E_{\PI}[f_\La]
\end{equation*}
for all bounded Borel $\La$, and all bounded measurable \emph{local} function $f$. Since local functions generate the $\sigma$-algebra on $\Conf$, one can then extend this result to all measurable functions by a monotone class argument. We obtain Theorem \ref{theo:DLR}.

\section{Translation-invariance: proof of Theorem \ref{theo:TI}}
\label{sec:ProofTI}
In this section we adapt the argument of Fröhlich-Pfister \cite{frohlich1981absence,frohlich1986absence} (see also \cite{MR1886241}) in order to prove translation-invariance of all Gibbs states (satisfying some technical conditions, see Section \ref{sec:good}). 

The basic idea is to construct “localized translations”: namely smooth, preferably area-preserving diffeomorphisms of $\R^2$ which act as true translations within a large region and are then “dampened” in order to coincide with the identity map outside some larger region. The point is then to evaluate the effect of such localized translations on the interaction energy. A striking fact is that if one can prove that the energy is shifted by at most a constant (either deterministically or in a strong enough probabilistic sense) under those localized translations, uniformly with respect to the size of the region, then it ensures that all Gibbs states are translation-invariant. Moreover, a classic trick in this field is to consider the \emph{average} effect of localized translations \emph{in two opposite directions} - in practice, this makes the energy shift easier to control while still being enough to reach the desired conclusion.

In the rest of Section \ref{sec:ProofTI}, we fix a solution $\Pp$ to \eqref{DLR} and a vector $\vu$ in $\R^2$ with $\|\vu\| \leq 1$ and we show (under some technical conditions, see Section \ref{sec:good}) that $\Pp$ is invariant by translation under $\vu$. Since $\vu$ is arbitrary, this clearly implies that $\Pp$ is translation-invariant. We write $\Ppp, \Ppm$ to denote the point processes obtained after translating $\Pp$ by $\pm \vu$.

\newcommand{\TLp}{\mathsf{T}^+_{L}}
\newcommand{\TLm}{\mathsf{T}^-_{L}}
\newcommand{\TLpm}{\mathsf{T}^\pm_{L}}
\newcommand{\psiLpm}{\psi_{L}^{\pm}}
\newcommand{\psiLp}{\psi_{L}^{+}}
\newcommand{\psiLm}{\psi_{L}^{-}}
\newcommand{\Tp}{\mathsf{T}^+}
\newcommand{\Tm}{\mathsf{T}^-}

\subsection{A construction of Georgii's}
\begin{lemma}[Existence of localized translations]
\label{lem:loc_trans} 
For $L \geq 1$, there exists two diffeomorphisms $\TLp, \TLm$ on $\R^2$ which are inverses of each other and such that:
\begin{enumerate}
    \item $\TLp, \TLm$ are area-preserving.
    \item We have:
    \begin{equation}
    \label{localized}
 \TLpm(x) = x \pm \vu \text{ when }|x| \leq L, \quad \TLpm(x) = x \text{ when } |x| \geq 2L.
    \end{equation}
    \item Letting $\psiLpm := \TLpm - \Id$ we have:
\begin{equation}
\label{eq:psipm}
|\psiLpm|_\kk \leq \Cc L^{-\kk}, \kk = 0, 1, 2, 3.
\end{equation}
\item We can write:
\begin{equation}
\label{eq:RemL}
\psiLp + \psiLm = \Rem_L, \text{ with } |\Rem_L|_{\kk}  \leq \Cc L^{-\kk-1}, \kk = 0, 1, 2.
\end{equation}
\end{enumerate}
\end{lemma}
The last identity \eqref{eq:RemL} expresses the fact that the localized translations $\TLp$ and $\TLm$ are almost exactly going “in opposite directions” (note that for $|x| \leq L$ this is the case, but between $|x| = L$ and $|x| = 2L$ they might not be perfectly aligned). Of course, one can easily construct two localized translations such that $\TLp$ and $\TLm$ are exactly going in opposite directions, but one of the two might then not be volume-preserving...

It will be useful to interpret the second bound in \eqref{eq:RemL} as the fact that $\Rem_L$ satisfies $|\Rem_L|_{\kk}  \leq \Cc L^{-1} \times L^{-\kk}$ with a “prefactor” of size $L^{-1}$.

\begin{proof}[Proof of Lemma \ref{lem:loc_trans}]
Without loss of generality, assume that $\vu = (0,1)$. We use the construction of \cite[Section 3.1]{MR1886241}, the starting point is to consider an Hamiltonian flow: let $h_1, h_2$ be two auxiliary functions in $C^{\infty}(\R)$ such that:
\begin{equation*}
\begin{cases}
h_1(s) = h_2'(s) = 1 & |s| \leq 1 \\
h_1(s) = h_2'(s) = 0 & |s| \geq 2,
\end{cases}
\end{equation*}
and to let $H : \R^2 \to \R^2$ be the “energy function” defined by $H(X) := h_1(x_1) h_2(x_2)$ with $X = (x_1, x_2)$ a point in $\R^2$. Let also $\vv$ be the associated Hamiltonian vector field, namely:
\begin{equation*}
\vv(X) = \nabla^{\perp} H(X) = (h_1(x_1)h'_2(x_2), - h'_1(x_1)h_2(x_2)).
\end{equation*}
For $L \geq 1$, let $\vv_L$ be the rescaled vector field 
\begin{equation*}
\vv_L(X) := \vv(L^{-1}X), \ X \in \R^2.
\end{equation*} 
It is observed in \cite{MR1886241} that the ODE $\dot{X} = \vv_L(X)$ defines a global flow $t \mapsto \tau_L(t)$. By Liouville's theorem, for all $t$ the map $\tau_L(t) : \R^2 \to \R^2$ is area-preserving. It realizes a localized translation by $t \vv$. We take:
\begin{equation*}
\TLp = \tau_L(1), \quad \TLm = \tau_L(-1),
\end{equation*}  
and since $\tau$ is a flow we have indeed $\TLp \circ \TLm = \TLm \circ \TLp = \Id$.

The controls on $\psiLpm$ and their derivatives follow from elementary but cumbersome computations, relying on \cite[Eq. 11]{MR1886241}, which is a simple scaling estimate:
\begin{equation}
\label{scalingvv}
|\vv_L(x) - \vv_L(y)| \leq \Cc L^{-1} |x-y|,
\end{equation}
for some constant $\Cc$ depending only on the choice of the auxiliary function $H$, and not on $L$. 

For example, write $\tau_L(t, \cdot)$ as $\Id + \psi_L(t, \cdot)$ and use the definition of the flow to compute (for $x, y$ in $\R^2$):
\begin{multline*}
\frac{\dd}{\dd t} |\psi_L(t,x) - \psi_L(t,y)|^2 = 2 \left\langle \psi_L(t,x) - \psi_L(t,y), \vv_L(\tau_L(t,x)) - \vv_L(\tau_L(t,y)) \right\rangle \\
\leq 2 \Cc L^{-1} |\psi_L(t,x) - \psi_L(t,y)| |x-y|,
\end{multline*}
which, since $\frac{\dd}{\dd t} |\psi_L(t,x) - \psi_L(t,y)|^2$ is also given by: 
\begin{equation*}
2 \frac{\dd}{\dd t} |\psi_L(t,x) - \psi_L(t,y)| \times |\psi_L(t,x) - \psi_L(t,y)|,
\end{equation*} 
yields:
\begin{equation*}
\frac{\dd}{\dd t} |\psi_L(t,x) - \psi_L(t,y)| \leq \Cc L^{-1} |x-y|.
\end{equation*}
In particular, we get $|\psi_L(t, \cdot)|_{\1} \leq \Cc L^{-1} |t|$ and taking $t = \pm 1$ yields $|\psiLpm|_\1 = \O(L^{-1})$. 

The other bounds in \eqref{eq:psipm} and the ones in \eqref{eq:RemL} are obtained similarly by deriving the relevant quantities with respect to time and using the ODE, before setting $t = \pm 1$.
\end{proof}

\subsection{Effect of the localized translations on the energy}
\newcommand{\Diff}{\mathsf{Diff}}
If $\Phi$ is a map $\Phi : \R^2 \to \R^2$ we use the notation $\Phi \cdot \bX$ to denote the push-forward of a measure (here a point configuration) $\bX$ by $\Phi$.

Fix $L$ and let $\La = \DD(10L)$. For $\bX$ in $\Conf$, define (with $\FLa$ as in \eqref{def:FLa}):
\begin{equation}
\label{eq:Diff_1}
\Diff^1_{\La}(\bX) := \hal \left( \FLa(\TLp \cdot \bX) + \FLa(\TLm \cdot \bX) \right) - \FLa(\bX),
\end{equation}
and for $\bX'$ in $\Conf(\La)$ such that $\Points(\bX', \La) = \Points(\bX, \La)$ let:
\begin{equation}
\label{eq:Diff_2}
\Diff^2_{\La}(\bX', \bX) := \hal \left( \MLa(\TLp \cdot \bX', \bX) + \MLa(\TLm \cdot \bX', \bX) \right) - \MLa(\bX', \bX).
\end{equation}
The next lemmas shows that $\Diff^1, \Diff^2$ are typically bounded \emph{uniformly with respect to} the size of $\La$. Let us emphasize that in contrast to the strategy used in \cite{frohlich1981absence,MR1886241} the effect of localized translations is here controlled in a probabilistic (and not deterministic) way.

\begin{lemma}
\label{lem:controlDiff1}
We have:
\begin{equation}
\label{eq:controlDiff1}
\log \E_{\PI} \left[ \exp\left( 2 \beta \left|\Diff^1_L(\bX)\right| \right) \right] \leq \Cc_\beta.
\end{equation}
\end{lemma}

\begin{lemma}
\label{lem:controlDiff2}
We have:
\begin{equation}
\label{eq:controlDiff2}
\log \E_{\PI} \left[ \exp\left(2 \beta \left|\Diff^2_L(\bX, \bX)\right| \right) \right] \leq \Cc_\beta.
\end{equation}
\end{lemma}

We start with the proof of Lemma \ref{lem:controlDiff1}.
\begin{proof}[Proof of Lemma \ref{lem:controlDiff1}]
The effect of a transport on the energy has been studied in details, but only for finite volume 2DOCP's. We must thus go through a finite-volume approximation. 

\emph{For simplicity, we will assume here that $\PI$ has been obtained as a limit point of $\Pp_{N, \mathsf{0}}$ - the local point process seen from the origin. The general case goes along the same lines, after translating each finite-$N$ quantity by $\x = \x(N)$. We only use below the finite volume results of Section \ref{sec:ResFinite}, which are valid everywhere in the bulk of $\LN$, see condition \eqref{condi:LL}.}

\paragraph{Step 1. Effect of a transport map on the energy - finite-volume.} Analyzing the effect of a transport map (thought of as a small perturbation of the identity map) on the Coulomb energy is a central and subtle task which was undertaken in \cite{leble2018fluctuations,MR4063572} and refined in \cite{serfaty2020gaussian}. We gather here some consequences of \cite[Prop.~4.2]{serfaty2020gaussian} combined with the local laws of \cite{Armstrong_2021}.

Let $\psi$ be a vector field on $\R^2$ and let $\Phi := \Id + \psi$. Assume that $\psi$ is $C^2$ and compactly supported in $\DD(x, L)$ where $(x, L)$ satisfies \eqref{condi:LL}. Assume that $|\psi|_\kk \leq \Cc_\psi L^{-\kk}$ for $\kk = 0, 1, 2$. If $L$ is large enough compared to $\Cc_\psi$, we have:
\begin{equation}
\label{eq:comparaisonFLa}
\FN(\Phi \cdot \bX_N) = \FN(\bX_N) + \Ani[\psi, \bXN] + \Error[\psi, \bXN],
\end{equation}
where $\Ani$ is the “anisotropy” or “angle” term which is linear in $\psi$ and can be thought of as a first-order term in the energy expansion, while $\Error$ is a second-order term\footnote{We will not go here into the definition of $\Ani$, and we use the results of \cite{serfaty2020gaussian} as a black box. We refer to the appendix of \cite{lambert2024law} for a “pedagogical” presentation.}. We have (if $|\psi|_\1, |\psi|_\2$ are small enough depending on $\beta$):
\begin{equation}
\label{eq:AniErrorExpMom}
\log \EN\left[  \exp\left( \beta |\Ani[\psi, \bXN]| \right) \right] \leq \Cc \times \Cc_\Psi \times L, \quad \log \EN\left[  \exp\left( \beta |\Error[\psi, \bXN]| \right) \right] \leq \Cc \times \Cc_\psi \times L^{-2} \log L.
\end{equation}
Using those results, we obtain:
\begin{claim}
\label{claim:EffetTransportFini}
We have:
\begin{equation}
\label{eq:effetTransportFini}
\log  \E_{\PNbeta}\left[ \exp\left( \beta \left| \FN(\bXN) - \hal \left(\FN(\TLp \cdot \bXN) + \FN(\TLm \cdot \bXN)\right)  \right|  \right)  \right] \leq \Cc_\beta.
\end{equation}
\end{claim}
\newcommand{\Rt}{\mathrm{R}_t}
\newcommand{\Rmt}{\mathrm{R}_{-t}}
\newcommand{\psimt}{\psi_{-t}}
\begin{proof}[Proof of Claim \ref{claim:EffetTransportFini}]
Write $\TLpm = \Id + \psiLpm$ as in Lemma \ref{lem:loc_trans}. Applying \eqref{eq:comparaisonFLa}, we get:
\begin{equation*}
\FN(\TLpm \cdot \bX) = \FN(\bX) + \Ani[\psiLpm, \bX] + \Error[\psiLpm, \bX].
\end{equation*}
Next, we combine the two expansions and use the linearity of $\Ani$:
\begin{multline*}
 \hal \left(\FN(\TLp \cdot \bXN) + \FN(\TLm \cdot \bXN)\right) - \FN(\bXN)  = \hal \left(\Ani[\psiLp + \psiLm, \bX] + \Error[\psiLp, \bX] + \Error[\psiLm, \bX]\right) \\
=  \hal \Ani\left[\Rem_L, \bX\right] + \hal \left(\Error[\psiLp, \bX] + \Error[\psiLm, \bX]\right),
\end{multline*}
with $\Rem_L$ as in \eqref{eq:RemL}. Using the control on exponential moments on $\Ani, \Error$ from \eqref{eq:AniErrorExpMom} and the bounds \eqref{eq:psipm}, \eqref{eq:RemL} on $\psiLpm$ and $\Rem_L$, we obtain \eqref{eq:effetTransportFini}.
\end{proof}

This gives us a control on the exponential moments of $\Diff^1$, but only when computed with respect to the interaction energy $\FN$ in the whole of $\LN$ - although the maps $\TLpm$ have an effect only at scale $L$. We now “localize” this estimate, which turns out to be a bit cumbersome, but will in fact prepare us for the proof of Lemma \ref{lem:controlDiff2}.

\newcommand{\bXNp}{\bXN^+}
\newcommand{\bXNm}{\bXN^-}
\newcommand{\Err}{\mathrm{Err}}

\paragraph{Step 2. Localization from $\FN$ to $\FLa$.}
A straightforward manipulation of the quadratic quantity defining the interaction energy yields:
\begin{multline}
\label{FiniteSize}
\FN(\bXN) - \hal \left(\FN(\TLp \cdot \bXN) + \FN(\TLm \cdot \bXN)\right) = \FLa(\bXN) - \hal \left(\FLa(\TLp \cdot \bXN) + \FLa(\TLm \cdot \bXN)\right) \\
+ \iint_{\La \times (\LN \setminus \La)} -\log|x-y| \dd \left(\bXN - \hal \left(\TLp \cdot \bXN + \TLm \cdot \bXN\right)\right)(x) \dd (\bXN - \Leb)(y).
\end{multline}
We now want to prove that the second line typically contributes to $\O(1)$, which will allow to transfer the information on the left-hand side (as stated in Claim \ref{claim:EffetTransportFini}) to the first line in the right-hand side, which is $\Diff^1_\La(\bXN)$. 

First, we make the following observation: by definition of a push-forward, we have
\begin{multline}
\label{Push}
\iint_{\La \times (\LN \setminus \La)} -\log|x-y| \dd \left(\bXN - \hal \left(\TLp \cdot \bXN + \TLm \cdot \bXN\right)\right)(x) \dd (\bXN - \Leb)(y) \\ \hspace{-0.2cm} = 
 \iint_{\La \times (\LN \setminus \La)} \left(- \log|x-y| - \hal \left(- \log|\TLp(x) -y | -  \log|\TLm(x) -y |\right)\right)  \dd \bXN(x) \dd (\bXN - \Leb)(y).
\end{multline}
Moreover, by construction, $\TLp$ and $\TLm$ are two area-preserving diffeomorphisms which coincide with the identity outside the disk $\DD(2L) \subset \La$, and hence we have, for any fixed $y$:
\begin{equation*}
\int_{\La} \left(- \log|x-y| - \hal \left(- \log|\TLp(x) -y | -  \log|\TLm(x) -y |\right)\right) \dd \Leb(x) = 0.
\end{equation*}
We may thus artificially introduce the background measure in the first integral (over $x \in \La$) of \eqref{Push} and re-write the second line of \eqref{FiniteSize}, \eqref{Push} as:
\begin{equation*}
\iint_{\La \times (\LN \setminus \La)} \left(- \log|x-y| - \hal \left(- \log|\TLp(x) -y| -  \log|\TLm(x) -y |\right)\right)  \dd \left(\bXN - \Leb\right)(x) \dd (\bXN - \Leb)(y).
\end{equation*}
Next, we use again the dyadic decomposition of space with respect to the $y$ variable: for $x,y$ in $\R^2$ and $i \geq 1$, define $\kappa_{i}(x,y)$ as:
\begin{equation}
\label{def:ki}
\kappa_{i}(x,y) \mapsto \left(- \log|x-y| - \hal \left(- \log|\TLp(x) -y | -  \log|\TLm(x) -y |\right)\right)    \chi_i(y).
\end{equation}
The second line of \eqref{FiniteSize}, \eqref{Push} can then be re-written as:
\begin{equation}
\label{dyadicki}
\sum_{i=1}^{+ \infty} \iint_{\LN \times \LN} \kappa_{i}(x,y)   \1_{\bLa}(y)  \dd \left(\bXN - \Leb\right)(x)  \dd (\bXN - \Leb)(y),
\end{equation} 
and our goal is now to prove the following, which in view of \eqref{FiniteSize} and combined with Claim \ref{claim:EffetTransportFini} would yield a finite-volume version of \eqref{eq:controlDiff1}:
\begin{claim}
\label{claim:FiniteSize2}
We have:
\begin{equation}
\label{eqfinitesize2}
\log  \E_{\PNbeta}\left[ \exp\left( 2 \beta \sum_{i=1}^{+ \infty} \left| \iint_{\LN \times \LN} \kappa_{i}(x,y) \1_{\bLa}(y)   \dd \left(\bXN - \Leb\right)(x)  \dd (\bXN - \Leb)(y) \right| \right) \right] \leq \Cc_\beta.
\end{equation}
\end{claim}
\begin{proof}[Proof of Claim \ref{claim:FiniteSize2}]
We first study $\kappa_{i}$ (see \eqref{def:ki}). Recall that:
\begin{itemize}
    \item The support of $\kappa_i$ with respect to the first variable is included in $\DD(2L)$.
    \item The support of $\kappa_i$ with respect to the second variable is included in the annulus $\DD(2^{i+1}) \setminus \DD(2^{i-1})$.
    \item For $y$ in $\bLa$ and $|x| \leq 2L$, we have $|x-y| \geq 8L$.
    \item $\TLpm(x) = x + \psiLpm(x)$, with $\psiLpm$ satisfying \eqref{eq:psipm} and \eqref{eq:RemL}.
    \item $|\chi_i|_\kk \leq \Cc_\chi 2^{-i\kk}$
\end{itemize}
After a Taylor's expansion of $\log|x-y + \psiLpm(x)|$, we get:
\begin{equation}
\label{taylor1}
\kappa_i(x,y) = \left( \hal \frac{\langle \Rem_L(x), x-y \rangle }{|x-y|^2} + \left(|\psiLp(x)|^2 + |\psiLm(x)|^2\right) \times \O(|x-y|^{-2}) \right)  \chi_i(y),
\end{equation}
where the $\O(|x-y|^{-2})$ can be differentiated. We expand $\frac{x-y}{|x-y|^2}$ once more around $x = 0$ and obtain:
\begin{equation}
\label{taylor2}
\frac{\langle \Rem_L(x), x-y \rangle }{|x-y|^2} = \frac{\langle \Rem_L(x), y \rangle }{|y|^2} + \Rem_L(x) |x| \times \O(|x-y|^{-2}),
\end{equation}
where the $\O(|x-y|^{-2})$ can be differentiated. 

Combining \eqref{taylor1} and \eqref{taylor2} and observing that $|\psiLp(x)|^2 + |\psiLm(x)|^2$ and $\Rem_L(x) |x|$ have the same size in the sense that:
\begin{itemize}
    \item $\psiLpm$ is of size $1$ and each derivative gains a factor $L^{-1}$ (\eqref{eq:psipm}),
    \item $\Rem_L$ is of size $L^{-1}$ and each derivative gains a factor $L^{-1}$ (\eqref{eq:RemL}), and $|x|$ is of size $L$,
\end{itemize}
we finally write $\kappa_i$ as $\kappa_{i,1} + \kappa_{i,2}$ where:
\begin{equation}
\label{eq:Kappa_iRewrite}
\kappa_{i,1}(x,y) = \hal \langle \Rem_L(x) , \frac{y}{|y|^2} \chi_i(y) \rangle, \quad \kappa_{i,2}(x,y) = \Err(x)\times \O(|x-y|^{-2})   \chi_i(y),
\end{equation}
with $|\Err|_\kk \leq \Cc L^{-\kk}$. In particular we have:
\begin{equation}
\label{eq:supykai}
\sup_{y} |\kappa_{i,2}(\cdot, y)|_{\1} \leq \Cc_\chi  L^{-1} 2^{-2i}, \quad |\kappa_{i,2}|_{\1 + \1} \leq \Cc_\chi L^{-1} 2^{-3i}.
\end{equation}
In the sequel, it will be useful to think of \eqref{eq:supykai} as the fact that:
\begin{itemize}
    \item For all $y$, $\kappa_{i,2}(\cdot, y)$ satisfies $|\kappa_{i,2}(\cdot, y)|_{\1} \leq \Cc_\chi 2^{-2i}  \times L^{-1}$ with a “prefactor” of size $2^{-2i}$.
    \item $\kappa_{i,2}$ satisfies $|\kappa_{i,2}|_{\1 + \1} \leq \Cc_\chi 2^{-2i} \times L^{-1} 2^{-i}$ with a “prefactor” of size $2^{-2i}$.
\end{itemize}

We now study the various layers within the dyadic decomposition of \eqref{dyadicki}. The indices $i$ such that $\supp \chi_i \subset \La$ do not play any role, and then we distinguish between the “first layer”, where $\supp \chi_i$ intersects $\partial \La$ (which corresponds to $i = \log_2 L + \O(1)$), the “last layer” $\supp \chi_i$ intersects $\partial \LN$ (which corresponds to $i = \hal \log_2 N + \O(1)$) and all the others “intermediate layers”. Let us start with the latter.

\paragraph{Intermediate layers.} Denote by $i_1$ the second layer (the first one appearing in the decomposition and not intersecting $\partial \La$), with $i_1 = \log_2 L + \O(1)$ and by $i_N$ the previous-to-last layer ($i_N = \hal \log_2 N + \O(1)$).
Since $\supp \chi_i$ does not intersect $\partial \La$ we may remove the indicator $\1_{\bLa}(y)$ and study:
\begin{equation*}
\EN\left[  \exp\left(2 \beta \sum_{i=i_1}^{i_N} \left|\iint_{\LN \times \LN} \kappa_{i}(x,y)    \dd \left(\bXN - \Leb\right)(x)  \dd (\bXN - \Leb)(y)\right| \right) \right] 
\end{equation*}
Splitting each $\kappa_i$ into $\kappa_{i,1} + \kappa_{i,2}$ and using Cauchy-Schwarz's inequality, it is enough to bound separately
\begin{multline*}
\EN \left[ \exp\left(\Cc \beta \sum_{i=_1}^{i_N} \left|\iint_{\LN \times \LN} \kappa_{i,1}(x,y)    \dd \left(\bXN - \Leb\right)(x)  \dd (\bXN - \Leb)(y)\right| \right) \right], \text{ and } \\ \EN \left[ \exp\left(\Cc \beta  \sum_{i=_1}^{i_N} \left|\iint_{\LN \times \LN} \kappa_{i,2}(x,y)    \dd \left(\bXN - \Leb\right)(x)  \dd (\bXN - \Leb)(y) \right| \right) \right].
\end{multline*}

\subparagraph{Contribution of $\kappa_{i,1}$.} 
We write, in view of \eqref{eq:Kappa_iRewrite}:
\begin{multline*}
\iint_{\LN \times \LN} \kappa_{i,1}(x,y)   \dd \left(\bXN - \Leb\right)(x)  \dd (\bXN - \Leb)(y) \\
= \left\langle \int_{\LN}  \hal \Rem_L(x)   \dd \left(\bXN - \Leb\right)(x), \int_{\LN}  \frac{y}{|y|^2} \chi_i(y)  \dd (\bXN - \Leb)(y) \right \rangle  = \hal \left\langle \Fluct[\Rem_L ] \times \Fluct[J_i] \right \rangle ,
\end{multline*}
using the notation $J_i : y \mapsto \frac{y}{|y|^2}$ as in \eqref{Ji}. 

Define $a_i := \Cc 2^{\frac{i}{10}} L^{-\frac{1}{10}}$ with $\Cc$ such that $\sum_{i=i_1}^{i_N} \frac{1}{a_i} =1$.
By convexity of $\exp$, we have:
\begin{multline}
\label{contributionki}
\EN \left[ \exp\left(\Cc \beta \sum_{i=i_1}^{i_N} \left| \iint_{\LN \times \LN} \kappa_{i,1}(x,y)    \dd \left(\bXN - \Leb\right)(x)  \dd (\bXN - \Leb)(y) \right| \right) \right] \\
= \EN \left[ \exp\left(\Cc \beta  \sum_{i=_i1}^{i_N} \frac{1}{a_i} a_i \left|\iint_{\LN \times \LN} \kappa_{i,1}(x,y)    \dd \left(\bXN - \Leb\right)(x)  \dd (\bXN - \Leb)(y) \right| \right) \right] \\
\leq \Cc \sum_{i=_u1}^{i_N} \frac{1}{a_i} \EN \left[ \exp\left(\Cc \beta a_i  \left|\iint_{\LN \times \LN} \kappa_{i,1}(x,y)    \dd \left(\bXN - \Leb\right)(x)  \dd (\bXN - \Leb)(y) \right| \right) \right],
\end{multline}
so we are left to control an average of terms of the form:
\begin{equation}
\label{eachterm}
\EN \left[ \exp\left(\Cc \beta   a_i |\Fluct[\Rem_L]| \times |\Fluct[J_i]| \right) \right], \quad i = \log_2 L + \O(1), \dots, \hal \log_2 N + \O(1).
\end{equation}
Before diving into the computations, let us present the heuristics: 
\begin{itemize}
    \item $J_i$ is a smooth function whose fluctuations are typically of order $2^{-i}$ by Lemma \ref{lem:fluct_smooth} (because we have $|J_i|_\kk \leq \Cc 2^{-i} \times 2^{-i\kk}$, with a “prefactor” of size $2^{-i}$, and $J_i$ is compactly supported within $\LN$).
    \item $\Rem_L $ is a Lipschitz function whose fluctuations are typically bounded by $1$ by Lemma \ref{lem:fluct_Lipschitz} (because we have $|\Rem_L |_\kk \leq \Cc L^{-1} L^{-\kk}$ with a “prefactor” of size $L^{-1}$),
\end{itemize}
thus $|\Fluct[\Rem_L]| \times |\Fluct[J_i]|$ is typically bounded by $2^{-i}$ and since we chose $a_i \leq 2^i$ we can expect each term \eqref{eachterm} to be of order at most $1$, which then yields a total contribution of order $1$ in \eqref{contributionki}. 

The difficulty is that we have a product of fluctuations in the exponent of \eqref{eachterm}, and our way out is that we can take $a_i$ much smaller than $2^{i}$ in \eqref{contributionki}. We argue as follows:
\begin{itemize}
    \item If $\left| \Fluct[\Rem_L ]\right| \leq 2^i a_i^{-1}$, since $|2^{i} J_i|_{\kk} \leq \Cc$ we can use the bound of Lemma \ref{lem:fluct_smooth}:
    \begin{equation}
    \label{eq:RemLBounded}
\EN \left[ \exp\left(\Cc \beta   a_i |\Fluct[\Rem_L | \times |\Fluct[J_i]| \right) \1_{\left| \Fluct[\Rem_L ]\right| \leq 2^i a_i^{-1}} \right] \\
\leq 2 \EN \left[ \exp\left( \beta  \Fluct[2^{i} J_i]  \right)  \right] \leq \Cc_\beta.
    \end{equation}

\item If $\left| \Fluct[\Rem_L ]\right| > 2^i a_i^{-1}$, we use instead Cauchy-Schwarz's inequality and write:
\begin{multline}
\label{RemLPasBoundedA}
\EN \left[ \exp\left(\Cc \beta   a_i | \Fluct[\Rem_L ] | \times | \Fluct[J_i] | \right) \1_{\left| \Fluct[\Rem_L ]\right| > 2^i a_i^{-1}} \right] \\
\leq \EN \left[ \exp\left( 2 \Cc \beta   a_i | \Fluct[\Rem_L ]| \times |\Fluct[J_i]|\right)  \right]^\hal \PNbeta\left[ \{\left| \Fluct[\Rem_L ]\right| > 2^i a_i^{-1}\} \right]^\hal.
\end{multline}
\begin{itemize}
    \item Observe that by \eqref{eq:psipm}, the function $L \Rem_L $ belongs to $\mathcal{L}_L$ with the notation of Lemma \ref{lem:fluct_Lipschitz}, hence 
    \begin{equation*}
\log \EN[e^{\Fluct^2[\Rem_L ]}] = \log \EN[e^{\frac{1}{L^2} \Fluct^2[L \Rem_L ]}] \leq \Cc_\beta L^{-2} \times L^2 \leq \Cc_\beta
    \end{equation*}
Since we chose $a_i$ such that $2^i a_i^{-1} \gg 1$, Markov's inequality yields:
\begin{equation}
\label{RemLPasBoundedB}
\PNbeta\left[ \{\left| \Fluct[\Rem_L ]\right| > 2^i a_i^{-1}\} \right]^\hal \leq \exp\left(- \frac{1}{\Cc_\beta} 2^{2i} a_i^{-2}\right).
\end{equation}
\item On the other hand, we can write, with $M \geq 1$ to be chosen later:
\begin{multline*}
a_i |\Fluct[\Rem_L ] | \times |\Fluct[J_i]| = \frac{a_i}{L 2^i} |\Fluct[L \Rem_L ] | \times |\Fluct[2^i J_i] | \\
\leq \frac{a_i}{L 2^i}  \left( \frac{M a_i^3}{L 2^{i}} \Fluct^2[L\Rem_L ] + \frac{L 2^{i}}{M a_i^3} \Fluct^2[2^i J_i] \right),
\end{multline*}
and thus, using Hölder's inequality and the first exponential moment bound of \eqref{eq:bound_fluct_Lipschitz}
\begin{equation}
\label{RemLPasBoundedC}
\EN \left[ \exp\left( 2 \Cc \beta   a_i | \Fluct[\Rem_L ]| \times |\Fluct[J_i]|\right)  \right] \leq \exp\left(\Cc_\beta \left(\frac{M a_i^4}{L^2 2^{2i}} \times L^2 + \frac{1}{M a_i^2} \times 2^{2i} \right)\right).
\end{equation}
We have used here the fact that, with our choice of $a_i$, we have $\frac{a_i^4}{L^2 2^{2i}} \ll 1, \frac{1}{a_i^2} \ll 1$ so we are indeed in position to apply \eqref{eq:bound_fluct_Lipschitz}.

\item Combining \eqref{RemLPasBoundedA}, \eqref{RemLPasBoundedB}, \eqref{RemLPasBoundedC} we obtain:
\begin{multline}
\label{RemLPasBoundedD}
\EN \left[ \exp\left(\Cc \beta   a_i | \Fluct[\Rem_L ] | \times | \Fluct[J_i] | \right) \1_{\left| \Fluct[\Rem_L ]\right| > 2^i a_i^{-1}} \right] \\
\leq \exp\left(\Cc_\beta \left(\frac{M a_i^4}{2^{2i}} + \frac{1}{M a_i^2} \times 2^{2i}\right) - \frac{1}{\Cc_\beta} 2^{2i} a_i^{-2} \right)
\end{multline}
by our choice of $a_i$ we have $\frac{a_i^4}{2^{2i}} \ll 2^{2i} a_i^{-2}$ and we can now choose $M$ large enough so that the negative term in the exponent wins. We obtain:
\begin{equation*}
\EN \left[ \exp\left(\Cc \beta   a_i | \Fluct[\Rem_L ] | \times | \Fluct[J_i] | \right) \1_{\left| \Fluct[\Rem_L ]\right| > 2^i a_i^{-1}} \right]  \leq \exp\left(-\frac{1}{\Cc_\beta} 2^{2i} a_i^{-2}\right) \leq 1.
\end{equation*}
\end{itemize}
\item In summary, we have shown that:
\begin{equation*}
\EN \left[ \exp\left(\Cc \beta   a_i  |\Fluct[\Rem_L ]| \times |\Fluct[J_i] \| \right) \right] \leq \Cc_\beta, 
\end{equation*}
and inserting this into \eqref{contributionki} yields:
\begin{equation}
\label{eq:HighLayersConclusion}
\EN \left[ \exp\left(\Cc \beta \sum_{i=i_1}^{i_N} \left| \iint_{\LN \times \LN} \kappa_{i,1}(x,y)    \dd \left(\bXN - \Leb\right)(x)  \dd (\bXN - \Leb)(y) \right| \right) \right] \leq \Cc_\beta,
\end{equation}
so the contribution of $\kappa_{i,1}$ for the “intermediate layers” is bounded.
\end{itemize}

\subparagraph{Contribution of $\kappa_{i,2}$.} Let us start again with some heuristics. We have $|\kappa_{i,2}|_{\1 + \1} \leq \Cc 2^{-2i} \times (L 2^{i})^{-1}$ (see \eqref{eq:supykai}) and thus by the bound on “bilinear statistics” stated in Lemma \ref{lem:bilinear}, each term of the form
\begin{equation*}
\iint_{\LN \times \LN} \kappa_{i,2}(x,y)   \dd \left(\bXN - \Leb\right)(x)  \dd (\bXN - \Leb)(y)
\end{equation*}
is typically of size at most $2^{-2i} \times L 2^{i} = L 2^{-i}$, so that the sum over $i$ should also be bounded (recall that $i \geq i_1 = \log_2 L + \O(1)$). 

To handle all the terms together, we first use Hölder's inequality and write e.g.
\begin{multline*}
\EN\left[\exp\left( 2\beta \sum_{i=i_1}^{i_N} \left| \iint_{\LN \times \LN} \kappa_{i,2}(x,y)   \dd \left(\bXN - \Leb\right)(x)  \dd (\bXN - \Leb)(y) \right| \right) \right] \\
= \EN\left[\exp\left(  2\beta \sum_{i=i_1}^{i_N}   2^{-\frac{i}{2}} L^{\hal}  2^{\frac{i}{2}} L^{-\hal} \left| \iint_{\LN \times \LN} \kappa_{i,2}(x,y)   \dd \left(\bXN - \Leb\right)(x)  \dd (\bXN - \Leb)(y) \right| \right) \right]  \\
\leq \sum_{i=i_1}^{i_N}   2^{-\frac{i}{2}} L^{\hal} \EN\left[\exp\left( 2\beta 2^{\frac{i}{2}} L^{-\hal} \left| \iint_{\LN \times \LN} \kappa_{i,2}(x,y)   \dd \left(\bXN - \Leb\right)(x)  \dd (\bXN - \Leb)(y) \right| \right) \right] \\
= 
\sum_{i=i_1}^{i_N}   2^{-\frac{i}{2}} L^{\hal} \EN\left[\exp\left( 2\beta 2^{-\frac{3i}{2}} L^{-\hal} \left| \iint_{\LN \times \LN} 2^{2i} \kappa_{i,2}(x,y)   \dd \left(\bXN - \Leb\right)(x)  \dd (\bXN - \Leb)(y) \right| \right) \right]
.
\end{multline*}
Now, since $|2^{2i} \kappa_{i,2}|_{\1 + \1} \leq \Cc_\chi L^{-1} 2^{-i}$ by \eqref{eq:supykai}  and $s := 2^{-\frac{3i}{2}} L^{-\hal}$ clearly satisfies
\begin{equation*}
2^{-\frac{3i}{2}} L^{-\hal} \leq \frac{L}{2^i}, 
\end{equation*}
we can apply Lemma \ref{lem:bilinear} and get (recall that $i \geq i_1 = \log_2 L + \O(1)$):
\begin{multline*}
\log \EN\left[\exp\left( 2\beta 2^{-\frac{3i}{2}} L^{-\hal} \left| \iint_{\LN \times \LN} 2^{2i} \kappa_{i,2}(x,y)  \dd \left(\bXN - \Leb\right)(x)  \dd (\bXN - \Leb)(y) \right| \right) \right] \\ \leq \Cc_\beta 2^{\frac{-3i}{2}} L^{-\hal} \times  2^i L \leq \Cc_\beta.
\end{multline*}

\subparagraph{Conclusion.} The contribution to the left-hand side of \eqref{eqfinitesize2} due to the “intermediate” layers is bounded. 

\paragraph{First layer.} Let us now turn to the “first” layer, where $\supp \chi_i$ intersects $\partial \La$ (which corresponds to $i = \log L + \O(1)$). We split again $\kappa_i$ as $\kappa_{i,1} + \kappa_{i,2}$.

\subparagraph{The contribution of $\kappa_{i,1}$.}
In view of \eqref{eq:Kappa_iRewrite} we can bound the contribution of $\kappa_{i,1}$ by:
\begin{multline*}
\log \EN\left[\exp\left( \iint_{\LN \times \LN} \kappa_{i,1}(x,y)   \dd \left(\bXN - \Leb\right)(x)  \dd (\bXN - \Leb)(y)  \right)\right] \\
\leq \log \EN\left[ \exp\left( |\Fluct[\Rem_L]| \times |\Fluct[J_i]| \right) \right] \\
\leq \log \EN\left[\exp\left( L^{-2} \Fluct^2[L \Rem_L] \right) \right] + \log \EN\left[ \exp \left(2^{-2i} \Fluct^2[2^i J_i] \right) \right].
\end{multline*}
We apply Lemma \ref{lem:fluct_Lipschitz} to both terms and control the sum by $\Cc_\beta$.

\subparagraph{The contribution of $\kappa_{i,2}$.}
We use a rough bound of the type:
    \begin{multline*}
\iint_{\LN \times \LN} \kappa_{i,2}(x,y) \1_{\bLa}(y)   \dd \left(\bXN - \Leb\right)(x)  \dd (\bXN - \Leb)(y) \\
 \leq \left( \Points(\bXN, \supp \chi_i) + |\supp \chi_i| \right) \times \sup_{y \in \supp \chi_i} \left| \int_{\LN} \kappa_{i,2}(x,y)   \dd \left(\bXN - \Leb\right)(x) \right|, 
    \end{multline*}
and then use Cauchy-Schwarz's inequality in the right-hand side to write:
\begin{multline*}
\left( \Points(\bXN, \supp \chi_i) + |\supp \chi_i| \right) \times \sup_{y \in \supp \chi_i} \left| \int_{\LN} \kappa_{i,2}(x,y)   \dd \left(\bXN - \Leb\right)(x) \right| \\
\leq \frac{2^{-2i}}{\Cc_\beta} \left(\Points^2(\bXN, \supp \chi_i) + 2^{2i} \right) + \Cc_\beta  2^{2i} \sup_{y \in \supp \chi_i} \Fluct^2[\kappa_{i,2}(\cdot, y)],
\end{multline*}
followed by Hölder's inequality (within the expectation):
\begin{multline*}
\log \EN\left[\exp \left( \beta \iint_{\LN \times \LN} \kappa_{i}(x,y) \1_{\bLa}(y)   \dd \left(\bXN - \Leb\right)(x)  \dd (\bXN - \Leb)(y) \right) \right] \\
\leq \log \EN\left[\exp \left( \frac{2^{-2i}}{\Cc_\beta} \left(\Points^2(\bXN, \supp \chi_i) + 2^{2i} \right) \right) \right] + \log \EN\left[ \sup_{y \in \supp \chi_i} \exp\left(\Cc_\beta  2^{2i}  \Fluct^2[\kappa_i(\cdot, y)] \right) \right].
\end{multline*}
The bound on exponential moments of $\Points^2$  in \eqref{eq:bound_points_finite} tells us that:
\begin{equation*}
 \log \EN\left[\exp \left( \frac{2^{-2i}}{\Cc_\beta} \left(\Points^2(\bXN, \supp \chi_i) + 2^{2i} \right) \right) \right] \leq \Cc_\beta, 
\end{equation*}
so we focus on the other term. We have, for all $y$, $|2^{2i} \kappa_{i,2}(\cdot, y)|_\1 \leq L^{-1}$ (see \eqref{eq:supykai}), and thus
    \begin{equation*}
\log \EN\left[ \sup_{y \in \supp \chi_i} \exp\left(\Cc_\beta  2^{2i}  \Fluct^2[\kappa_{i,2}(\cdot, y)] \right) \right] \leq \log \EN\left[ \sup_{\varphi \in \mathcal{L}_{L}} \exp\left(\Cc_\beta 2^{-2i} \Fluct^2[ \varphi ] \right) \right]
    \end{equation*}
    where $\mathcal{L}_{L}$ is as in Lemma \ref{lem:fluct_Lipschitz}. Applying that lemma in its strongest version (the first bound in \eqref{eq:bound_fluct_Lipschitz}) gives:
    \begin{equation*}
 \log \EN\left[ \sup_{\varphi \in \mathcal{L}_{L}} \exp\left(\Cc_\beta 2^{-2i} \Fluct^2[ \varphi ] \right) \right] \leq \Cc_\beta 2^{-2i} L^2 \leq \Cc_\beta
    \end{equation*}
since $i = \log_2 L + \O(1)$.

The last layer can be treated the same way. This concludes the proof of Claim \ref{claim:FiniteSize2}.
\end{proof}

\subsubsection*{Step 3. Conclusion: proof of \eqref{eq:controlDiff1}.}
Combining \eqref{eq:effetTransportFini}, the identity \eqref{FiniteSize} and  Claim \ref{claim:FiniteSize2}, we deduce that:
\begin{equation}
\label{eq:FiniteSize3}
\log  \E_{\PNbeta}\left[ \exp\left( \beta \left| \FLa(\bXN) - \hal \left(\FLa(\TLp \cdot \bXN) + \FLa(\TLm \cdot \bXN)\right)  \right|  \right) \right] \leq \Cc_\beta,
\end{equation}
which expresses the fact that \emph{the effect of an (averaged) localized translation on the local interaction energy in $\La$ has bounded exponential moments under the finite-volume Gibbs measure.} 

In order to pass \eqref{eq:FiniteSize3} to the exp-tame limit, which would allow us to derive \eqref{eq:controlDiff1}, we need to check that the function $\bX \mapsto \beta \left| \FLa(\bX) - \hal \left(\FLa(\TLp \cdot \bX) + \FLa(\TLm \cdot \bX)\right)  \right|$ is indeed exp-tame. It is certainly local, so it remains to check that it does not grow faster than $\frac{1}{\Cc_\beta} \Points(\bX, \La)^2$.

Observe that since $|\psiLpm|_\1 \leq \Cc L^{-1}$ (see \eqref{eq:psipm}) we have, for $L$ large enough, that $|x-y|$ and $|\TLpm(x) - \TLpm(y)|$ are always comparable, and in particular:
\begin{equation*}
\inf_{t \in [0,1]} \left|x-y + t(\psiLpm(x) - \psiLpm(y))\right| \geq \hal |x-y|.
\end{equation*}
It implies that we may write a Taylor's expansion of $\log\left|x-y + \psiLpm(x) - \psiLpm(y)\right|$ of the form:
\begin{equation*}
\log\left|x-y + \psiLpm(x) - \psiLpm(y)\right| = \log|x-y| + \O\left(\frac{|\psiLpm(x) - \psiLpm(y)|}{|x-y|}\right) = \log |x-y| + \O(L^{-1}).
\end{equation*}
Hence with a crude bound we obtain indeed:
\begin{multline*}
\FLa(\bX) - \hal \left(\FLa(\TLp \cdot \bX) + \FLa(\TLm \cdot \bX) \right) \\
= \iint_{\La \times \La} -\log|x-y| \dd \left(\bX - \hal \left(\TLp \cdot \bX + \TLm \cdot \bX\right) \right)(x) \dd \left(\bX - \hal \left(\TLp \cdot \bX + \TLm \cdot \bX\right) \right) (y) \\ 
= \iint_{\La \times \La} \left(- \log|x-y| - \hal \left(- \log|\TLp(x) -y | -  \log|\TLm(x) -y |\right)\right) \dd \left(\bX  - \Leb\right)(x) \dd \left(\bX  - \Leb\right)(y) \\
= \iint_{\La \times \La} \O(L^{-1}) \dd \left(\bX  - \Leb\right)(x) \dd \left(\bX  - \Leb\right)(y)  \leq \Cc L^{-1} \left( |\La|^2 + \Points^2(\bX, \La) \right).
\end{multline*}
Thus $\Diff^1$ is indeed $\exp$-tame, which concludes the proof of Lemma \ref{lem:controlDiff1}.
\end{proof}

It remains to prove Lemma \ref{lem:controlDiff2}. As it turns out, this was essentially done in the course of the previous proof.
\begin{proof}[Proof of Lemma \ref{lem:controlDiff2}]
 Observe that, by definition, $\Diff^2_\La(\bX, \bX)$ is given by (cf. \eqref{Push}):
\begin{multline*}
\iint_{\La \times (\R^2 \setminus \La)} -\log|x-y| \dd \left(\bX - \hal \left(\TLp \cdot \bX + \TLm \cdot \bX \right)\right)(x) \dd (\bX - \Leb)(y) \\ = 
\iint_{\La \times (\R^2 \setminus \La)} \left(- \log|x-y| - \hal \left(- \log|\TLp(x) -y | -  \log|\TLm(x) -y |\right)\right)  \dd \bX(x) \dd (\bX - \Leb)(y),
\end{multline*}
and that for the same reason as above we may introduce the background in the first integral. We must thus control (cf. \eqref{dyadicki}) the exponential moments of:
\begin{multline*}
\iint_{\La \times (\R^2 \setminus \La)} \left(- \log|x-y| - \hal \left(- \log|\TLp(x) -y | -  \log|\TLm(x) -y |\right)\right)  \dd \left(\bX- \Leb\right)(x) \dd (\bX - \Leb)(y) \\
= \sum_{i=1}^{+\infty} \iint_{\R^2 \times \R^2} \kappa_{i}(x,y) \1_{\bLa}(y) \dd \left(\bX- \Leb\right)(x) \dd (\bX - \Leb)(y).
\end{multline*}
This is as in \eqref{dyadicki} but for an infinitely extended exterior configuration. Since we know that the results of Lemma \ref{lem:fluct_Lipschitz}, Lemma \ref{lem:fluct_smooth} etc. are valid for the infinite-volume limit as well as for the finite-$N$ system, the proof of Claim \ref{claim:FiniteSize2} applies readily. 
\end{proof}

\subsection{Definition of “good” processes}
\label{sec:good}
We are now in position to define our class of “good” point processes. We say that $\Pp$ is good when it satisfies the following properties for some constant $\Cc$ \emph{possibly depending on $\Pp$}.
\begin{enumerate}
    \item Control on the exponential moment of the number of points as in \eqref{ControlPointsInfini}:
\begin{equation*}
\log \E_{\Pp} \left[\exp\left( \frac{1}{\Cc} \Points^2(\bX, \DD(x,\ell)) \right)  \right] \leq \Cc \ell^4.
    \end{equation*}
    \item Control on the $1$-point correlation function as in \eqref{bound:CF_infinite_record}:
    \begin{equation*}
\sup_{y \in \R^2} \rho_{1, \Pp}(y) \leq \Cc.
    \end{equation*}
    \item Control on the exponential moment of fluctuations for Lipschitz test functions as in \eqref{eq:bound_Fluct_LIPINF}
\begin{equation}
\log \E_{\Pp} \left[ \sup_{\varphi \in \Lell} e^{\frac{1}{\Cc_\beta} \Fluct^2[\varphi]}  \right] \leq \Cc \ell^2,
\end{equation}
and for smooth test functions as in \eqref{eq:bound_Fluct_smoothINF}.    
    \item Control on the exponential moment of $\Diff^1, \Diff^2$ as in \eqref{eq:controlDiff1}, \eqref{eq:controlDiff2}:
    \begin{equation}
    \label{diffForGood}
\log \E_{\Pp} \left[ \exp\left( 2 \beta \Diff^1_L(\bX) \right) \right] \leq \Cc, \quad \log \E_{\Pp} \left[\exp\left(2\beta \Diff^2_L(\bX, \bX) \right) \right] \leq \Cc.
    \end{equation}
\end{enumerate}
We can already make some observations:
\begin{enumerate}
    \item Every weak limit point $\PI$ is “good”.
    \item Following the proof of Proposition \ref{prop:Zero} (see Remark \ref{rem:NeededForMove}), we know that the first three items guarantee that for all compact $\La \subset \R^2$, for $\Pp$-a.e. $\bX$ and for all configuration $\bX'$ in $\La$ \emph{such that $\bX'$ has the same number of points as $\bX$ in $\La$}, the quantity $\MLa(\bX', \bX)$ exists and is finite as in \eqref{eq:MLa}, hence \eqref{DLR} equations make sense.
    \item The mixture of two “good” processes is good.
\end{enumerate}

\subsection{The infinite-volume argument}
\newcommand{\GG}{\mathcal{G}}
\subsubsection*{Reminder on the geometry of Gibbs measures.}
The trick of \cite{bricmont1979equivalence} used in \cite{frohlich1981absence,MR1886241} (and many other subsequent works) relies on the structure of extremal Gibbs measures. 

Let $\GG$ be the subset of all probability measures $\Pp$ on $\Conf$ such that:
\begin{enumerate}
    \item $\Pp$ is “good” in the sense of Section \ref{sec:good}.
    \item $\Pp$ satisfies \eqref{DLR} (which makes sense thanks to the first item).
\end{enumerate}
$\GG$ clearly forms a convex subset and one can consider its extremal elements. Every measure in $\GG$ has an extremal decomposition i.e. can be written (in an essentially unique way) as a mixture of extremal ones. Extremal Gibbs measures enjoy interesting properties, in particular it is known (see e.g. \cite[Lemma 6.62.]{MR3752129}) that:
\begin{equation}
\label{ExtremalPpty} \text{Two distinct extremal elements of $\GG$ are mutually singular.}
\end{equation}
Let us quickly comment on the proof of \eqref{ExtremalPpty}, as we are imposing extra conditions by considering only “good” solutions to \eqref{DLR}, which is not quite the usual setting. 

\begin{proof}[Proof of \eqref{ExtremalPpty}] We follow the proof of \cite[Thm. 6.58]{MR3752129}.
The key fact leading to \eqref{ExtremalPpty} is that extremal elements of $\GG$ are tail-trivial: if $A$ is an event which is measurable with respect to $\R^2 \setminus \DD(R)$ for all $R > 0$, then $\Pp(A) = 0$ or $1$ for $\Pp$ extremal. The reason for this is that otherwise, one could consider $\Pp_0 := \frac{1}{\Pp(A)} \1_{A} \Pp$ and $\Pp_1 := \frac{1}{1 - \Pp(A)} \1_{\bar{A}} \Pp$ and write $\Pp$ as the mixture of $\Pp_0, \Pp_1$, which would contradict the extremality of $\Pp$ provided $\Pp_0$ and $\Pp_1$ are both in $\GG$.

Now, the fact that $\Pp_0, \Pp_1$ both satisfy \eqref{DLR} is a general fact that uses only the structure of DLR equations, see \cite[Prop. 6.61]{MR3752129}. It remains to observe that $\Pp_0, \Pp_1$ are also both “good” processes: this is true because we allowed the constant $\Cc$ in the definition of “good” processes to depend on $\Pp$.

Since two solutions of \eqref{DLR} which coincide on tail events are always equal, this concludes.
\end{proof}

Theorem \ref{theo:TI} can thus be reduced to showing that (recall that $\Ppp, \Ppm$ are the translates of $\Pp$ by $\pm \vu$):
\begin{proposition}
\label{prop:reduceTI}
Every $\Pp \in \GG$ is absolutely continuous with respect to $\hal\left( \Ppp + \Ppm \right)$.
\end{proposition}
Indeed, to prove Theorem \ref{theo:TI}, namely that all elements of $\GG$ are translation-invariant, it is enough to show that all \emph{extremal} elements of $\GG$ are translation-invariant (as this property is stable under mixtures). Observe that for all $\La$ and all $\bX, \bX'$, since the interaction potential is translation-invariant, we have:
\begin{equation*}
\FLa(\bX) = \F_{\La + \u}(\bX + \u), \quad \MLa(\bX', \bX) = \mathsf{M}_{\La + \u}(\bX' + \u, \bX + \u),
\end{equation*}
from which it follows that if $\Pp$ satisfies \eqref{DLR}, so do $\Ppp$ and $\Ppm$. Moreover it is clear that “goodness” of a point process is invariant by translation, hence if $\Pp$ is in $\GG$ so are $\Ppp$ and $\Ppm$ - and if $\Pp$ is extremal, so are they. 

To conclude, let $\Pp$ be an extremal element of $\GG$ and assume the result of Proposition \ref{prop:reduceTI}. Then $\Pp$ is absolutely continuous with respect to a mixture of $\Ppp$ and $\Ppm$, and thus cannot be singular with both of them. We must have $\Pp = \Ppp$ or $\Pp = \Ppm$, and the two are in fact equivalent.

We now turn to proving Proposition \ref{prop:reduceTI}. The argument is similar in spirit to the classical ones of \cite{frohlich1981absence,MR1886241}, however here the energy cost of localized translations is not bounded uniformly but probabilistically (see Lemma \ref{lem:controlDiff1} and Lemma \ref{lem:controlDiff2}) which induces some extra manipulations.

\begin{proof}[Proof of Proposition \ref{prop:reduceTI}]
If we prove that for all events $A$ in $\Conf$, we have:
\begin{equation}
\label{AbsoCont}
\Pp^2(A) \leq \Cc \hal\left( \Ppp(A) + \Ppm(A) \right)
\end{equation}
with some constant $\Cc$ depending only on $\Pp$, then absolute continuity of $\Pp$ with respect to $\hal\left( \Ppp + \Ppm \right)$ will follow. In fact, it is enough to show \eqref{AbsoCont} for all events $A$ that are \emph{local} and to apply a monotone class argument: the set of all events $A$ such that \eqref{AbsoCont} holds clearly forms a monotone class, and contains local events, thus it contains the $\sigma$-algebra generated by them, which is the whole $\sigma$-algebra.

Let $A$ be a local event in $\Conf$ and $L \geq 1$ be such that $A$ is $\DD(\frac{L}{10})$-local. We will simply write $\Tp, \Tm$ as well as $\Diff^1, \Diff^2$ instead of $\Diff^1_{\La}, \Diff^2_\La$ when referring to the objects defined in \eqref{eq:Diff_1}, \eqref{eq:Diff_2}.

On the one hand, Cauchy-Schwarz's inequality gives:
\begin{multline}
\label{PpPpPm}
\Pp(A)^2 = \E_\Pp[\1_A]^2 = \E_\Pp\left[\1_A e^{-\frac{\beta}{2} \left(\Diff^1(\bX) + \Diff^2(\bX, \bX)\right)} e^{\frac{\beta}{2} \left(\Diff_1(\bX) + \Diff_2(\bX, \bX)\right)} \right]^2 \\
\leq \E_\Pp\left[\1_A e^{- \beta \left(\Diff_1(\bX) + \Diff_2(\bX, \bX) \right)} \right] \E_\Pp\left[e^{\beta \left(\Diff_1(\bX) + \Diff_2(\bX, \bX)\right)} \right], 
\end{multline}
and on the other hand, using the \eqref{DLR} equations and the definition of $\Diff^1, \Diff^2$:
\begin{multline*}
\E_\Pp\left[\1_A e^{- \beta \left(\Diff_1(\bX) + \Diff_2(\bX, \bX) \right)} \right] \\
= \E_\Pp\left[ \frac{1}{\KLbeta(\bX)} \int_{\bX' \in \Conf(\La, \Points(\bX, \La))} \1_A(\bX') e^{- \beta \left(\Diff^1(\bX') + \Diff^2(\bX', \bX) \right)} e^{- \beta\left( \FLa(\bX') + \MLa(\bX', \bX)    \right) } \dd \bX'  \right] \\
= \E_\Pp\left[ \frac{1}{\KLbeta(\bX)}  \int_{\bX' \in \Conf(\La, \Points(\bX, \La))} \1_A(\bX') e^{- \hal \beta\left( \FLa(\Tp \cdot \bX') + \MLa(\Tp \cdot \bX', \bX)  \right) - \hal \beta\left( \FLa(\Tm \cdot \bX') + \MLa(\Tm \cdot \bX', \bX)  \right) } \dd \bX'   \right].
\end{multline*}
For clarity, here and below we will write $\dd \bX'$ and add the condition $\bX' \in \Conf(\La, \Points(\bX, \La))$ instead of writing $\dd \Bin_{\La, \Points(\bX, \La)}(\bX')$.

Using the convexity of $\exp$, we obtain:
\begin{multline}
\label{usingConvexity}
\E_\Pp\left[\1_A e^{- \beta \left(\Diff_1(\bX) + \Diff_2(\bX, \bX) \right)} \right] \\
\leq \hal \E_\Pp\left[ \frac{1}{\KLbeta(\bX)} \int_{\bX' \in \Conf(\La, \Points(\bX, \La))} \1_A(\bX') e^{- \beta\left( \FLa(\Tp \cdot \bX') + \MLa(\Tp \cdot \bX', \bX)  \right)} \dd \bX'   \right] \\ 
+ \hal \E_\Pp\left[ \frac{1}{\KLbeta(\bX)}  \int_{\bX' \in \Conf(\La, \Points(\bX, \La))} \1_A(\bX') e^{- \beta\left( \FLa(\Tm \cdot \bX') + \MLa(\Tm \cdot \bX', \bX)  \right) } \dd \bX'   \right].
\end{multline}

Next, we change variables in both integrands and write, for fixed $\bX$:
\begin{multline*}
\int_{\bX' \in \Conf(\La, \Points(\bX, \La))} \1_A(\bX') e^{- \beta\left( \FLa(\Tp \cdot \bX') + \MLa(\Tp \cdot \bX', \bX)  \right)  } \dd \bX' \\
= \int_{\bX' \in \Conf(\La, \Points(\bX, \La))} \1_A(\Tm \cdot \bX') e^{- \beta\left( \FLa(\bX') + \MLa(\bX', \bX)  \right)  } \dd \bX'
\end{multline*}
and similarly for the other term. We have used here that $\Tp, \Tm$ are inverses of each other, are both area-preserving, and that applying $\Tp$ or $\Tm$ preserves the number of points in $\La$ since they map $\La$ to itself bijectively.

Since by construction $\Tm$ acts as a true translation on $\DD(\frac{L}{2})$ and since we assumed that $A$ is $\DD(\frac{L}{10})$-local, we have $\1_A(\Tm \cdot \bX') = \1_A(\bX' - \vu)$, hence using \eqref{DLR} equations and the definition of $\Ppm$ we get:
\begin{multline}
\label{CSDiffA}
\E_{\Pp}\left[ \frac{1}{\KLbeta(\bX)} \int_{\bX' \in \Conf(\La, \Points(\bX, \La))} \1_A(\Tm \cdot \bX') e^{- \beta\left( \FLa(\bX') + \MLa(\bX', \bX)  \right) } \dd \bX'  \right] \\
= \E_{\Pp} \left[ \frac{1}{\KLbeta(\bX)}  \int_{\bX' \in \Conf(\La, \Points(\bX, \La)} \1_A(\bX' - \vu) e^{- \beta\left( \FLa(\bX') + \MLa(\bX', \bX)  \right) } \dd \bX'  \right] \\
= \E_{\Pp}[\1_A(\bX' - \vu)] = \Ppm(A).
\end{multline}
A similar computation can be done with $\Tp$ instead of $\Tm$. We obtain:
\begin{equation*}
\E_\Pp\left[\1_A e^{- \beta \left(\Diff_1(\bX) + \Diff_2(\bX, \bX) \right)} \right] \leq \hal \left( \Ppm(A) + \Ppp(A) \right).
\end{equation*}

It remains to bound the second term in the right-hand side of \eqref{PpPpPm} by a constant. To do so, we rely on our assumption that $\Pp$ is “good” and apply \eqref{diffForGood}. This concludes the proof of Theorem \ref{theo:TI}.

\end{proof}
\appendix

\section{Reminders on the electric field formalism}
\label{sec:ElecFormalism}
The proof of Lemmas \ref{lem:fluct_Lipschitz}, \ref{lem:LipschitzHardWall} and \ref{lem:SmoothHardWall} require the use of the “electric” formalism for 2DOCP's as developed by Serfaty et al. We give here a concise presentation and refer to the forthcoming book \cite{serfaty2024lectures} for a deeper analysis.

\subsection{Electric fields and truncations}
We recall that $-\log$ is, up to a multiplicative constant, the solution of the Laplace equation in $\R^2$:
\begin{equation*}
- \Delta (-\log) = 2\pi \delta_0.
\end{equation*}

\begin{definition}[True electric potential and electric field]
\label{def:true}
Let $\bXN$ be a finite point configuration. We define the \textit{electric potential} $\HH^{\bXN}$ generated by $\bXN$ as the scalar field on $\R^2$ defined by:
\begin{equation*}
\HH^{\bXN}(x) := \int_{\LN} - \log |x-y| \dd (\bXN - \Leb)(y).
\end{equation*}
We let $\nHH^{\bXN}$ be the \textit{electric field} generated by $\bXN$, namely the vector field defined on $\R^2$ by:
\begin{equation*}
\nHH^{\bXN}(x) = \int_{\LN} - \nabla \log |x-y| \dd (\bXN - \Leb)(y) = \nabla \HH^{\bXN}.
\end{equation*}
\end{definition}
Note that the electric field satisfies:
\begin{equation}
\label{divnhh}
- \dive \nHH^{\bXN}(x) = 2\pi \left(\bXN - \1_{\LN} \Leb \right).
\end{equation}

\begin{definition}[Truncation and spreading out Dirac masses]
For $\eta > 0$ we let $\feta$ be the function:
\begin{equation*}
    \ff_\eta(x) := \min\left( \log \frac{|x|}{\eta}, 0\right).
\end{equation*}
For each point $x$ of $\bXN$, let $\eta(x)$ be a non-negative real number. The data of $\vec{\eta} = \{\eta(x), x \in \bXN\}$ is called a \textit{truncation vector}. We let $\HH^{\bXN}_{\veta}$ (resp. $\nHH^{\bXN}_{\veta}$) be the \textit{truncated} electric potential (resp. field) given by:
\begin{equation*}
\HH^{\bXN}_{\veta} := \HH^{\bXN} - \sum_{x \in \bXN} \ff_{\eta(x)}(\cdot - x), \quad \text{resp. } \nHH^{\bXN}_{\veta} = \nHH^{\bXN} - \sum_{x \in \bXN} \nabla \ff_{\eta(x)}(\cdot - x).
\end{equation*}
It has the effect of truncating the singularity near each point charge by \textit{spreading out} the point charge $\delta_x$.

Let $\delta^{(\eta)}_x$ be the uniform probability measure on the the circle of center $x$ and radius $\eta$, and let $\bXN^{(\eta)}$ be the spread-out measure
\begin{equation}
\label{measurespreadout}
\bXN^{(\veta)} := \sum_{x \in \bXN} \delta^{(\eta)}_x.
\end{equation}
The truncated electric field satisfies (cf. \eqref{divnhh}):
\begin{equation}
\label{divnhheta}
- \dive \nHH_{\veta}^{\bXN}(x) = 2\pi \left(\bXN^{(\veta)} - \1_{\LN} \Leb \right).
\end{equation}
\end{definition}

\begin{definition}[Nearest-neighbor distances]
Let $\bXN$ be a point configuration in $\LN$. For every point $x$ of $\bXN$ we define the “nearest-neighbor” distance $\rr(x)$ as:
\begin{equation}
\label{def:nn_distance}
\rr(x) := \frac{1}{4} \min \left( \min_{y \in \bXN, y \neq x} |x-y|, 1 \right).
\end{equation}
In particular $\rr(x)$ is always smaller than $1$. 
\end{definition}

\subsection{Global and local electric energy}
The following identity may be considered as the starting point of an approach to Coulomb gases based on electric fields and their “renormalized (electric) energy”.
\begin{proposition}[Logarithmic interaction energy as renormalized electric energy]
We see $\bXN$ as a point configuration in $\R^2$ and let $\HH^{\bXN}$ be the true electric potential generated by $\bXN$. Taking for each $x$ a truncation $\eta(x) \leq \rr(x)$, it holds:
\begin{equation*}
\FN(\bXN) = \hal \left( \frac{1}{2\pi} \int_{\R^2} |\nHH^{\bXN}_{\veta}|^2 + \sum_{x \in \bX} \log \eta(x)\right)  - \sum_{x \in \bX} \int_{\DD(x,\eta(x))} \ff_{\eta(x)}(t-x) \dd t.
\end{equation*} 
\end{proposition}

\newcommand{\Ener}{\mathrm{Ener}}
We then define a notion of \emph{local electric energy} by setting, if $\Omega$ is a disk (or a square...):
\begin{equation*}
\Ener(\bXN, \Omega) := \int_{\Omega} |\nHH^{\bXN}_{\vec{\rr}}|^2.
\end{equation*}

The “local laws” of \cite{Armstrong_2021} ensure that the electric energy is always of the same order as the area of $\Omega$, in exponential moments.
\begin{proposition}[Local laws for the electric energy]
\label{prop:LL}
Let $x$ be a point in $\R^2$ and $\ell > 0$ be some lengthscale such that $(x, \ell)$ satisfies \eqref{condi:LL}. We have:
\begin{equation}
\label{eq:LocalLaw}
\log \EN\left[ \exp\left( \frac{\beta}{2} \Ener(\bXN, \DD(x, \ell)) \right) \right] \leq \Cc_\beta \ell^2.
\end{equation}
\end{proposition}
In the sequel, we use the notation $\EnerPts$:
\begin{equation*}
\EnerPts(\bXN, \Omega) := |\Omega|^\hal \Ener(\bXN, \Omega)^\hal + \Points(\bXN, \Omega).
\end{equation*}
Note that by the local laws these two terms are of the same order $\O(|\Omega|)$ in exponential moments.

\section{Control on Lipschitz test functions}
\subsection{Deterministic bounds on fluctuations in terms of electric energy}
\label{sec:FluctEnergy}
The following lemma is the key for the controls on fluctuations of Lipschitz linear statistics stated in Section \ref{sec:ResFinite}.
\begin{lemma}
\label{lem:apriori}
\

\textbf{One variable.} Let $\varphi$ be a function in $C^1(\R^2)$ with compact support in $\DD(L)$. Let $\Omega$ be a disk containing a $1$-neighborhood of $\DD(L)$. We have the deterministic estimate:
    \begin{equation}
    \label{eq:apriori_one_var}
    \left|\int \varphi(x) \dd \left(\bXN - \1_{\LN} \Leb  \right)(x) \right| \leq \Cc |\varphi|_{\1} \times  \EnerPts(\bXN, \Omega).
    \end{equation}

\textbf{Two variables.} Let $(x,y) \mapsto \Ka(x,y)$ be a function in $C^2(\R^2 \times \R^2)$ with compact support in $\DD(L_1) \times \DD(L_2)$. Let $\Omega_1, \Omega_2$ be disks containing a $1$-neighborhood of $\DD(L_1), \DD(L_2)$. We have the deterministic estimate:
    \begin{multline}
    \label{eq:apriori_two_var}
    \left|\iint_{\R^2 \times \R^2} \Ka(x,y) \dd \left(\bXN - \1_{\LN} \Leb  \right)(x) \dd \left(\bXN - \1_{\LN} \Leb  \right)(y) \right| \\
    \leq \Cc |\Ka|_{\1 + \1} 
    \times  \EnerPts\left(\bXN, \Omega_2\right) \times \EnerPts\left(\bXN, \Omega_1\right).
    \end{multline}
\end{lemma}
\begin{proof}[Proof of Lemma \ref{lem:apriori}]
For a given point $x$ of the point configuration $\bX$ we define the truncation $\eta(x)$ as 
\begin{equation*}
\eta(x) = \begin{cases}
0 &  \text{if } x \notin \supp \nabla \varphi \\
\rr(x) & \text{if } x \in \supp \nabla \varphi,
\end{cases}
\end{equation*}
where $\rr$ denotes the “nearest-neighbor” distance as in \eqref{def:nn_distance}. Let us recall that, by definition, we always take $|\rr| \leq 1$. We replace each Dirac mass $\delta_x$ for $x \in \bXN$ by its “spread-out” version $\delta_{x}^{\eta(x)}$ as in \eqref{measurespreadout} and we let $\bXN^{(\veta)}$ be the “spread-out” configuration.

The measures $\bX_{\veta}$ and $\bX$ coincide outside a $1$-neighborhood of $\supp \varphi$  and the points have been spread out at a distance at most $1$, thus we may write:
\begin{equation}
\label{erreurReg}
\left|\int_{\R^2}\varphi(x) \dd \left(\bXN(x) - \bXN^{(\veta)}\right)(x) \right| \leq \Cc |\varphi|_{\1} \times \Points(\bX, \Omega).
\end{equation}

Using first the triangle inequality and then the identity \eqref{divnhheta} we may write:
\begin{multline*}
\left|\int \varphi(x) \dd \left(\bXN - \1_{\LN} \Leb  \right)(x) \right| \leq \left|\int_{\R^2} \varphi(x) \dd \left(\bXN^{\veta} - \Leb \right)(x) \right| + \Cc |\varphi|_{\1} \times \Points(\bX, \Omega) \\
\leq \Cc \left| \int_{\R^2} \varphi(x) \dive \nHH_{\veta}^{\bXN} \right| + \Cc |\varphi|_{\1} \times \Points(\bX, \Omega).
\end{multline*}
Integrating by parts and using Cauchy-Schwarz's inequality yields
\begin{multline*}
- \int_{\R^2} \varphi(x) \dive \nHH_{\veta}^{\bXN} = \int_{\supp \nabla \varphi} \nabla \varphi \cdot \nHH_{\veta}^{\bXN} \leq \left(\int_{\supp \nabla \varphi} |\nabla \varphi|^2 \right)^\hal \left(\int_{\supp \nabla \varphi} \left|\nHH_{\veta}^{\bXN}\right|^2 \right)^\hal \\
 \leq |\varphi|_{\1} |\supp \nabla \varphi|^{\hal} \times  \left(\int_{\supp \nabla \varphi} \left|\nHH_{\veta}^{\bXN}\right|^2 \right)^\hal,
\end{multline*}
which combined with \eqref{erreurReg} concludes the proof of \eqref{eq:apriori_one_var}.

The estimate for test functions of two variables is deduced by applying the previous control twice. We introduce an auxiliary map $A$ defined by:
\begin{equation*}
A : y \mapsto \int_{\LN} \Ka(x,y) \dd \left(\bXN- \Leb\right)(x).
\end{equation*}
This is clearly in $C^1_c(\R^2)$, and the quantity we are interested is nothing but 
\begin{equation*}
\int_{\LN} A(y) \dd \left(\bXN-\Leb\right)(x)
\end{equation*}
 We may apply the control \eqref{eq:apriori_one_var} and get:
\begin{equation*}
 \left|\iint_{\R^2 \times \R^2} \Ka(x,y) \dd \left(\bXN - \1_{\LN} \Leb  \right)(x) \dd \left(\bXN - \1_{\LN} \Leb  \right)(y) \right| \leq \Cc |A|_{\1} \EnerPts(\bXN, \Omega_2).
\end{equation*}
We are left to estimate the quantity $|A|_{\1}$. To do that, let us write:
\begin{equation*}
\nabla A(y) := \int_{\LN} \nabla_y \Ka(x,y) \dd \left(\bXN-\Leb\right)(x).
\end{equation*}
We are again in the position of applying \eqref{eq:apriori_one_var}:
\begin{equation*}
\left|\nabla A(y)\right| \leq \Cc |\Ka|_{1+1} \EnerPts(\bXN, \Omega_1).
\end{equation*}
This yields \eqref{eq:apriori_two_var}.
\end{proof}

Note that in the proof of Lemma \ref{lem:apriori} we never used the fact that the test functions were supported within $\LN$ or not.

\subsection{Application: proof of Lemmas \ref{lem:fluct_Lipschitz}, \ref{lem:bilinear}, \ref{lem:LipschitzHardWall}}
To prove Lemma \ref{lem:fluct_Lipschitz} we use \eqref{eq:apriori_one_var} and write:
\begin{multline*}
\log \EN \left[ \sup_{\varphi \in \Lell} e^{\frac{1}{\Cc_\beta} \Fluct^2[\varphi]}  \right] 
\leq \log \EN \left[ \sup_{\varphi \in \Lell} e^{\frac{1}{\Cc_\beta} \Cc^2 |\varphi|^2_{\1} \EnerPts(\bXN, \Omega)^2 } \right] \\
\leq \log \EN \left[ \sup_{\varphi \in \Lell} e^{\frac{1}{\Cc_\beta} \Cc^2 \ell^{-2} \left( \ell^2 \Ener(\bXN, \Omega) + \Points(\bXN, \Omega)^2 \right) } \right] \\
\leq \log \EN \left[ e^{\frac{\Cc^2}{\Cc_\beta} \Ener(\bXN, \Omega)} \right] + \log \EN \left[ e^{\frac{\Cc^2}{\Cc_\beta} \ell^{-2} \Points(\bXN, \Omega)^2} \right],
\end{multline*}
and combine this with the local laws \eqref{eq:LocalLaw} for the energy and \eqref{eq:bound_points_finite} for the number of points. The proof of Lemma \ref{lem:LipschitzHardWall} is the same. For the proof of Lemma \ref{lem:bilinear}, we use \eqref{eq:apriori_two_var} instead
\begin{multline*}
\log \EN \left[ \sup_{\Ka \in \LLLL} \exp\left( s \iint_{\LN \times \LN} \Ka(x,y) \dd \left(\bXN - \Leb  \right)(x) \dd \left(\bXN - \Leb  \right)(y) \right)  \right] 
\\
\leq \log \EN \left[ \sup_{\Ka \in \LLLL} \exp\left(  \Cc |s| |\Ka|_{\1 + \1} \times  \EnerPts\left(\bXN, \Omega_2\right) \times \EnerPts\left(\bXN, \Omega_1\right) \right) \right] \\
\leq \log \EN \left[ \exp\left(  \Cc |s| L_1^{-1} L_2^{-1} \times  \left( \frac{L^2_1}{L^2_2} \EnerPts^2\left(\bXN, \Omega_2\right) + \frac{L^2_2}{L^2_1} \EnerPts^2\left(\bXN, \Omega_1\right) \right) \right) \right] 
\\
\leq \log \EN \left[ \exp\left(  \Cc |s| \frac{L_1}{L_2^3} \EnerPts^2\left(\bXN, \Omega_2\right) \right) \right] + \log \EN \left[ \exp\left(  \Cc |s|\frac{L_2}{L_1^3} \EnerPts^2\left(\bXN, \Omega_1\right) \right) \right]. 
\end{multline*}
Since $|s| \leq \frac{L_1}{L_2}$ we may apply \eqref{eq:LocalLaw} (and Hölder's inequality) to control both terms by $\Cc_\beta |s| L_1L_2$.

\section{Radial linear statistics hitting the hard wall: proof of Lemma \ref{HardWallSmooth}}
\label{sec:ProofHardWall}
Let $\delta > 0$ to be chosen later, we can write, using the identity \eqref{divnhh}:
\begin{equation}
\label{eq:introdelta}
\Fluct[\varphi] = \frac{1}{2\pi} \int \nabla \varphi \cdot \nHH = \frac{1}{2\pi} \int_{\{\dist(x, \partial \LN) > \delta \sqrt{N} \}} \nabla \varphi \cdot \nHH + \frac{1}{2\pi} \int_{\{\dist(x, \partial \LN) \leq \delta \sqrt{N}\}} \nabla \varphi \cdot \nHH.
\end{equation}

Let us focus on the first term in the right-hand side of \eqref{eq:introdelta}. Using the rotational symmetry, write $\varphi(x) = \bphi(|x|)$ for some $\bphi : [0, + \infty) \to \R$, and work in polar coordinates:
\begin{equation*}
\frac{1}{2\pi}  \int_{\{\dist(x, \partial \La_N) > \delta\}} \nabla \varphi \cdot \nHH = \int_{r = 0}^{+\infty} \1_{|r - \RN| \geq \delta \sqrt{N}} \ \bphi'(r) \left( \frac{1}{2\pi}  \int_{x \in \partial \DD(0, r)} \nHH \cdot \vec{n} \right) \dd r,
\end{equation*}
where $\vec{n}$ is the outer normal to the disk $\DD(0, r)$. By definition of the discrepancy, we have
\begin{equation*}
\frac{1}{2\pi}  \int_{x \in \partial \DD(0, r)} \nHH \cdot \vec{n}= \Di(\bX, \DD_r)
\end{equation*}
and thus (using the fact that the discrepancy is zero outside $\LN$ as the system is globally neutral):
\begin{multline*}
\int_{\{\dist(x, \partial \LN) > \delta\}} \nabla \varphi \cdot \nHH = \int_{r = 0}^{+\infty} \1_{|r - \RN| \geq \delta \sqrt{N}}  \bphi'(r) \Di(\bX, \DD_r) \dd r = \int_{r = 0}^{\RN -  \delta \sqrt{N}}  \bphi'(r) \Di(\bX, \DD_r) \dd r \\
\leq |\varphi|_\1 \times \int_{r = 0}^{\RN -  \delta \sqrt{N}} |\Di(\bX, \DD_r)| \dd r.
\end{multline*}
The usual controls on $\Di$ found on \cite{Armstrong_2021} are not enough, we rely instead of the hyperuniformity result of \cite{leble2021two}, which guarantees that as long as $r = r(N)$ tends to infinity with $N$ while satisfying $r(N) \leq \RN - \delta \sqrt{N}$, we have: 
\begin{equation*}
\lim_{N \to \infty} \frac{\EN[|\Di(\bX, \DD_r)|] }{r} = 0,
\end{equation*}
but with a speed of convergence that depends on $\delta$. Since $|\varphi|_\1 = \O(\sqrt{N})$ by assumption, this ensures, for $\delta$ fixed, that:
\begin{equation}
\label{loindubord}
\lim_{N \to \infty} \frac{1}{N} \int_{\{\dist(x, \partial \LN) > \delta\}} \nabla \varphi \cdot \nHH = 0.
\end{equation}

We now study the second term in the right-hand side of \eqref{eq:introdelta}. First, we regularize the electric field with some arbitrary truncation $\eta > 0$ and write:
\begin{multline*}
\int_{\{\dist(x, \partial \LN) \leq \delta \sqrt{N}\}} \nabla \varphi \cdot \nHH  \leq |\varphi|_\1 \int_{\{\dist(x, \partial \LN) \leq \delta \sqrt{N}\}} |\nHH_{\veta}| + |\varphi|_\1 \times N \times |\feta|_{L^1} \\
\preceq |\varphi|_\1 \times \left(\sqrt{N} \times \delta \sqrt{N}\right)^\hal \times \left( \int_{\R^2} \left|\nHH_{\veta}\right|^2 \right) + |\varphi|_\1 \times N  \times o_\eta(1).
\end{multline*}
The global law for the electric energy ensures that $\EN\left[ \int_{\R^2}  \left|\nHH_{\veta}\right|^2  \right] \leq \Cc(\beta, \eta) \times N$. We thus obtain:
\begin{equation}
\label{presdubord}
\EN\left[  \int_{\{\dist(x, \partial \La_N) \leq \delta \sqrt{N}\}} \nabla \varphi \cdot \nHH \right] \leq \delta^\hal \Cc(\beta, \eta) \times |\varphi|_\1 \times N + |\varphi|_\1 \times N  \times o_\eta(1).
\end{equation}

Combining \eqref{eq:introdelta}, \eqref{loindubord} and \eqref{presdubord} and choosing first $\eta$ small enough, then $\delta$ small enough (depending on $\eta$) and finally sending $N$ to infinity, we see that $\limsup_{N \to \infty} \frac{1}{\sqrt{N}} \EN [|\Fluct[\varphi]|]$ can be made arbitrarily small, which concludes.

\section{Upgrading the convergence to exp-tame: proof of Lemma \ref{lem:upgrading}}
\label{sec:ProofUpgrading}
\begin{proof} We first need bounds of the type \eqref{bound:CF}  on correlation functions for the infinite-volume processes.

\textbf{1. Passing correlation upper bounds to the limit.}  It is mentioned in \cite[Rem. 1.4]{thoma2022overcrowding} that any infinite-volume weak limit point $\PI$ will enjoy the conclusion of Lemma \ref{lem:bound_CF}, let us sketch the argument here.

The correlation functions are defined (see \eqref{def:rho_k}) by using indicator functions of disks, but they can easily be bounded using smoother test functions: for all $r > 0$ we have:
\begin{equation}
\label{chimrrhok}
\EN\left[ \sum_{i_1 \neq i_2 \dots \neq i_k} \prod_{l=1}^k \chimr(x_{i_l} - y_l) \right]
\leq  \|\rho_{k, \PNbeta}\|_{L^\infty} \times r^{2k}
\end{equation}
and on the other hand, for all $y_1, \dots, y_k$:

\begin{equation}
\label{rhokchim}
\rho_{k, \PNbeta}(y_1, \dots, y_k) \leq \liminf_{r \to 0} \frac{1}{\pirk} \EN\left[ \sum_{i_1 \neq i_2 \dots \neq i_k} \prod_{l=1}^k \chipr(x_{i_l} - y_l) \right]. 
\end{equation}
For all $r > 0$ and all fixed $y_1, \dots, y_k$ in $\R^2$, the map
\begin{equation*}
\bX \mapsto \sum_{x_{i_1} \neq x_{i_2} \dots \neq x_{i_k} \in \bX} \prod_{l=1}^k \chipr(x_{i_l} - y_l)
\end{equation*}
is continuous on $\Conf$ and bounded. We may thus pass to the weak limit as $N \to \infty$ for fixed $r$ before sending $r \to 0$, using \eqref{bound:CF}, \eqref{chimrrhok} and \eqref{rhokchim}. It yields (for some constants depending on $k$ and $\beta$):
\begin{equation}
\label{bound:CF_infinite}
\sup_{y_1, \dots, y_k} \rho_{k, \PI}(y_1, \dots, y_k) \leq \Cc_k \|\rho_{k, \PNbeta}\|_{L^\infty} \leq \Cc'_k.
\end{equation} 

\newcommand{\tfk}{\tilde{f}_k}
\newcommand{\tfkep}{\tilde{f}_{k, \epsilon}}
\newcommand{\fkep}{f_{k, \epsilon}}

\textbf{2. An approximation argument.}
Let $f : \Conf \to \R$ be a local and exp-tame test function. Without loss of generality, we can assume that $f$ is $\La$-local for some disk $\La$. In order to pass to the weak limit as $N \to \infty$, we want to approximate $f$ by a continuous function on $\Conf$. 

Let $\epsilon, \delta > 0$. Let $d_\epsilon : \R^2 \to [0,1]$ be a smooth cut-off function equal to $1$ on a $\epsilon$-neighborhood of $\partial \La$ and to $0$ outside a $2\epsilon$-neighborhood of $\partial \La$. The corresponding linear statistics $T_\epsilon : \Conf \to [0, + \infty)$ defined as:
\begin{equation*}
T_\epsilon(\bX) := \sum_{x \in \bX} d_\epsilon(x)
\end{equation*}
is continuous on $\Conf$ by definition of the vague topology. Note that $T_\epsilon$ is larger than the number of points of $\bX$ in the $\epsilon$-neighborhood of $\partial \La$.

 Additionally, let $I : [0, + \infty) \to [0,1]$ be a smooth version of $x \mapsto \ind_{x = 0}$ obtained by setting $I(x) = 1$ if $|x| \leq \frac{1}{100}$ and $I(x) = 0$ if $|x| \geq \frac{1}{10}$. 

 The map $\bX \mapsto I \circ T_\epsilon(\bX)$ is continuous on $\Conf$ as the composition of continuous maps, and bounded. By construction: 
\begin{multline}
\label{PptyITep}
\begin{cases}
\text{ There is a point of $\bX$ in a $\epsilon$-neighborhood of $\partial \La$} & \implies I \circ T_\epsilon(\bX) = 0 \\ 
\text{ There is no point of $\bX$ in a $2\epsilon$-neighborhood of $\partial \La$} & \implies I \circ T_\epsilon(\bX) = 1.
\end{cases}
\end{multline}

Notice that for any $k \geq 0$, the map $\bX \mapsto \ind_{\Points(\bX, \La) = k}$ is not continuous for the weak topology (because point can move in/out through the boundary), but the map $\bX \mapsto \ind_{\Points(\bX, \La) = k} I \circ T_\epsilon(\bX)$ is continuous, because points that approach the boundary of $\partial \La$ will now be detected by $I \circ T_\epsilon$.

We now proceed in two steps:

First, observe (using the second implication in \eqref{PptyITep}) that, for some constant $\Cc_\La$ depending on (the perimeter of) $\La$:
\begin{multline*}
\left|\EN[f] - \EN[f I \circ T_\epsilon]\right| \leq |f|_{\0} \times \PNbeta\left( \text{There is at least one point of $\bX$ in a $2\epsilon$-neighborhood of $\partial \La$} \right)\\
 \leq |f|_{\0} \times \Cc_\La \epsilon \times \|\rho_{1, \PNbeta} \|_{L^\infty},
\end{multline*}
and a similar estimate holds for $\PI$. Since $\sup_{N \geq 1} \|\rho_{1, \PNbeta} \|_{L^\infty}$ and $\|\rho_{1, \PI} \|_{L^\infty}$ are finite, we obtain for some constant $\Cc$ depending on $\La, f$ etc. but uniform in $N$: 
\begin{equation}
\label{IntroduireMollif}
\left|\EN[f] - \E_{\PI}[f]\right| \leq \left|\EN[f I \circ T_\epsilon] - \E_{\PI}[f I \circ T_\epsilon]\right| + \Cc \epsilon,
\end{equation}

Next, by definition of an exp-tame function, and in view of Lemma \ref{lem:points} (finite-volume) and Lemma \ref{lem:ControlPointsInfini} (infinite-volume), there exists $K \geq 1$ such that:
\begin{equation}
\label{queuesdespoints}
\limsup_{N \to \infty} \E_{\PpNx}[|f(\bX)| \ind_{\Points(\bX, \La) \geq K}] \leq \epsilon, \quad \E_{\PI}[|f(\bX)| \ind_{\Points(\bX, \La) \geq K}] \leq \epsilon.
\end{equation}

We now proceed to the approximation argument in itself. For $k = 0, \dots, K-1$, let $f_k$ be the function $f_k := f \ind_{\Points(\bX, \La) = k}$. It can be seen as a measurable map $\tfk$ on $\La^k$, with $\tfk(x_1, \dots, x_k) := f_k(\sum_{i=1}^k \delta_{x_i})$, which is bounded ($|\tfk| \leq e^{\frac{1}{\Ccb}k^2}$, by definition of exp-tame functions). For each $k$, we may thus find a continuous function $\tfkep : \La^k \to \R$ such that $\|\tfkep - \tfk\|_{L^1(\La^k)} \leq \frac{\epsilon}{K}$. Up to making an arbitrarily small error in $L^1(\La^k)$ we can assume that $\tfkep$ is compactly supported within $\La^k$, and extend it continuously by $0$ on $(\R^2)^k \setminus \La^k$, and then turn it back into a continuous function $\fkep$ on $\Conf$ by setting:
\begin{equation*}
\fkep(\bX) := \frac{1}{k!} \sum_{x_1 \neq \dots \neq x_k \in \bX} \tfkep(x_1, \dots, x_k)
\end{equation*}
Note that, by construction and by definition of the correlation functions, we have:
\begin{equation*}
\E_{\PpNx}\left[ \left|f(\bX) \ind_{\Points(\bX, \La) = k}  - \fkep(\bX) \ind_{\Points(\bX, \La) = k} \right| \right] \leq \|\rho_{k, \PNbeta}\|_{L^{\infty}} \times \|\tfkep - \tfk\|_{L^1(\La^k)} \leq \|\rho_{k, \PNbeta}\|_{L^{\infty}} \times  \frac{\epsilon}{K}.
\end{equation*}
As a consequence, we obtain that the map defined on $\Conf$ by:
\begin{equation*}
\bX \mapsto \fkep(\bX) \ind_{\Points(\bX, \La) = k} I \circ T_\epsilon(\bX)
\end{equation*}
is continuous, bounded, and satisfies:
\begin{equation*}
\E_{\PpNx}\left[ \left|f(\bX) \ind_{\Points(\bX, \La) = k} I \circ T_\epsilon(\bX)  - \fkep(\bX) \ind_{\Points(\bX, \La) = k}  I \circ T_\epsilon(\bX) \right| \right] \leq \|\rho_{k, \PNbeta}\|_{L^{\infty}} \times  \frac{\epsilon}{K}.
\end{equation*}
Summing over $k = 0, \dots, K-1$ we may thus form a continuous, bounded function $f_\epsilon$ on $\Conf$ such that:
\begin{equation*}
\left|\E_{\PpNx}\left[\left(f(\bX)  I \circ T_\epsilon(\bX)  - f_\epsilon(\bX)\right) \ind_{\Points(\bX, \La) < K} \right] \right| \leq \Cc_K \epsilon, 
\end{equation*}
and, with the same argument, $\left|\E_{\PI}\left[ \left(f(\bX)  I \circ T_\epsilon(\bX)  - f_\epsilon(\bX)\ind_{\Points(\bX, \La) < K}\right) \right] \right| \leq \Cc_K \epsilon$. 

Combined with \eqref{IntroduireMollif} and \eqref{queuesdespoints}, and passing $\E_{\PpNx}[f_\epsilon]$ to the weak limit as $N \to \infty$, we obtain that $\EN[f]$ and $\E_{\PI}[f]$ are arbitrarily close, which concludes the proof.
\end{proof}
\clearpage

\bibliographystyle{alpha}
\bibliography{DLR2d}

\end{document}